\documentclass[11pt, a4paper, oneside]{article}
\usepackage[utf8]{inputenc}
\usepackage[T1]{fontenc}
\usepackage[a4paper,top=2cm,bottom=2cm,left=2cm,right=2cm,marginparwidth=1.75cm]{geometry}
\usepackage{enumitem}
\usepackage{ mathrsfs }
\usepackage{url}
\usepackage{tikz}
\usepackage{tikz-cd}
\usepackage{amsmath}
\usepackage{amsthm}
\usepackage{pgfplots}
\usepackage{graphicx}
\usepackage{comment}
\usepackage{subcaption}
\usepgflibrary{shapes.misc} 
\pgfplotsset{width=11cm,compat=1.15}
\usepgfplotslibrary{fillbetween}
\usepackage{amsfonts}
\usepackage{amssymb}
\usepackage{mathabx}
\usepackage{graphicx}
\usepackage{xfrac}
\usepackage{comment}
\usepackage{mathtools}
\usepackage{pgfplots}

\newtheorem{theorem}{Theorem}[section]

\newtheorem{lemma}[theorem]{Lemma}
\newtheorem{prop}[theorem]{Proposition}
\newtheorem{rmk}[theorem]{Remark}
\newtheorem{ex}[theorem]{Example}
\newtheorem{definition}[theorem]{Definition}

\newtheorem{clm}[theorem]{Claim}

\def    \C      {{\mathbb C}}
\def    \R      {{\mathbb R}}

\title{On singular Lagrangian fibrations and applications to symplectic embeddings I}

\author{Santiago Achig-Andrango, Renato Vianna and Alejandro Vicente}

\AtEndDocument{\bigskip{\footnotesize
  \textsc{Santiago Achig-Andrango, Universidad San Francisco de Quito, Ecuador.} \par  
  \textit{E-mail}:  \texttt{ eachig@usfq.edu.ec}\\
  
  \textsc{Renato Vianna, IME-USP, Brazil.} \par  
  \textit{E-mail}:  \texttt{ renato.vianna@ime.usp.br}\\
  
   \textsc{Alejandro Vicente, University of Stavanger, Norway.} \par  
  \textit{E-mail}:  \texttt{ kvicente931207@gmail.com}
  }}

\date{}

\begin{document}
\maketitle
\begin{abstract}
	In this paper, we construct singular Lagrangian fibrations on some examples of disk cotangent bundles in dimensions 4 and 6. As an application, we show how this construction can be used to obtain toric domains in some cases. In particular, we recover results from Ferreira, Ramos, and Vicente on the Gromov width of the disk cotangent bundle of spheres of revolution, as well as results from Ramos on the Lagrangian bidisk. We also briefly discuss how this technique can be used to study symplectic embedding problems.
\end{abstract}

\section{Introduction}\label{sec: intro}

Lagrangian fibrations are objects of paramount importance in Symplectic Topology. They have emerged gradually through several mathematical developments, with key ideas forming at the intersection of classical mechanics, symplectic geometry, and algebraic geometry. The beginning of the realization of the concept of Lagrangian fibration came from the study of integrable systems in the late 18th and early 19th century, by Joseph-Louis Lagrange and William Rowan Hamilton. However, one of the earliest clear instances where Lagrangian fibrations appear explicitly is in the work of Joseph Liouville in the mid-19th century, see \cite{Liouville1855}. All this came together as part of a formal structure in the mid-20th century, in the works of Vladimir Arnold, Jean-Marie Souriau, and others. This marked the birth of Symplectic Geometry as a formal framework. 

Basically, when a symplectic manifold
$X$ admits a Lagrangian torus fibration $f : X\to B$, an integral affine structure is induced on the base B, from which we can infer a lot of information
about $X$. Lagrangian fibrations touch upon many of the active lines of research in the area. Perhaps they are recognized the most, by their role in Mirror Symmetry, particularly in the SYZ conjecture. There is, however, another line of work that has increasingly witnessed, in the last three decades, the incursion of Lagrangian fibrations: the study of symplectic embeddings. Works of McDuff, Schlenk and Traynor (see \cite{Latschev2011TheGW} and \cite{Traynor1995SymplecticPC}), pioneered the use of toric Lagrangian fibrations for studying symplectic embeddings. Recently, their line of ideas has significantly evolved in the works of Ramos, Ostrover, Sepe, Ferreira and Vicente (see \cite{ramos2017symplectic}, \cite{Ramos2017OnTR}, \cite{Ferreira2021SymplecticEI}, \cite{Ferreira2023GromovWO}, and \cite{Ostrover2023FromLP}); with the use of integrable systems techniques in order to construct toric Lagrangian fibrations in a wide class of symplectic manifolds. These toric Lagrangian fibrations we refer to are a class of toric domains. A \textit{toric domain} is a subset of $\C^n$ defined by
\[\mathbb{X}_\Omega=\left\{(z_1,\ldots,z_n)\in \C^n\mid(\pi|z_1|^2,\ldots,\pi|z_n|^2)\in\Omega\right\},\] where $\Omega\subset\R_{\ge 0}^n$ is the closure of an open set in $\R^n$. Toric domains are, in particular, examples of Lagrangian fibrations with only elliptic singularities. It turns out that some symplectic manifolds are toric domains in disguise, meaning that, after finding suitable action-angle coordinates, they can be shown to be symplectomorphic to a toric domain. In these cases, finding the description of the corresponding $\Omega\subset\R_{\ge 0}^n$ has useful applications to the study of symplectic embeddings. We expand on this ahead.

This paper focuses on constructing singular Lagrangian fibrations (see Definition \ref{defsinglagfib}) on certain disk cotangent bundles of dimensions 4 and 6. We also outline how to use these singular Lagrangian fibrations to, on one hand, recover the toric domains produced in \cite{ramos2017symplectic} and \cite{Ferreira2023GromovWO}, while on the other hand, how to use them to study symplectic embedding problems. In particular, we discuss lower bounds of the Gromov width for the examples studied here, by means of using the Traynor trick on our fibrations.

These singular Lagrangian fibrations are constructed using the elementary
reduction procedure summarized in Proposition~\ref{generalmeth2}. Roughly
speaking, let \(M^n\) be a Riemannian manifold and let \(D^*M\) be its disk
cotangent bundle. Assume that \(D^*M\) carries a Hamiltonian
\(T^{n-1}\)-action with moment map
\[
\mu:D^*M\to \mathbb R^{n-1}.
\]
For a regular value \(\lambda\) of \(\mu\), the quotient
\[
(D^*M\cap \mu^{-1}(\lambda))/T^{n-1}
\]
is a two-dimensional reduced space. If this reduced space is identified with a
surface \(X_\lambda^2\), and if one chooses a smooth family of simple closed
curves in \(X_\lambda^2\), then the preimage of each such curve is a
Lagrangian torus in \(D^*M\).

The role of the curves is especially important in our setting because we work
with disk cotangent bundles, hence with symplectic domains with boundary.
The curves are chosen so that their preimages remain inside the disk cotangent
bundle. In the examples below, these curves are used to define an action
coordinate, given by the symplectic area of a disk whose boundary lies on the
corresponding Lagrangian torus. Thus the resulting fibration is described on
the regular part by a map of the form
\[
(\mu,A):D^*M\to \mathbb R^n.
\]

The singular fibers arise when the chosen curves meet the locus in the reduced
spaces corresponding to non-free torus orbits. In the examples considered in
this paper, this singular locus is explicit and controlled; this is what allows
us to identify the nodal singularities and the corresponding monodromy.\\

The approach developed in this paper can be generalized to a setting where we require only a $T^k$-action, producing $n-k$ functions that measure areas of disks on a $T^{n-k}$ torus. However, the conditions needed to ensure that the resulting fibration is Lagrangian become substantially more complicated when $k<n-1$. A more general theory along these lines is currently under development by the authors.

\begin{rmk}
Symplectic manifolds admitting a Hamiltonian $T^{n-1}$-action, and more generally a Hamiltonian $T^k$-action, are known, under suitable additional conditions, as complexity-one and complexity-$(n-k)$ spaces. These have been widely studied in several works, for instance, \cite{KarshonTolman2001} and \cite{KarshonTolman2014}. The relation between these spaces and Lagrangian fibrations have been studied in \cite{HohlochKarshon2021} and the references therein. We emphasize though, that the manifolds studied in this note are manifolds with boundary, so special care should be taken when producing the desired singular Lagrangian fibrations. This is the role of the foliation by the curves $\gamma_s$ above.
\end{rmk}

\subsection{Singular Lagrangian fibration constructions} \label{sec: rescslf}

The following results in the construction of singular Lagrangian fibrations are direct applications of our technique exposed in Proposition~\ref{generalmeth2}.

Let \(S\subset\mathbb R^3\) be a sphere of revolution. We write
\begin{equation}\label{def: surf_rev}
    S=
\left\{
(x,y,z)\in\mathbb R^3\mid x^2+y^2=u(z)^2,\ z\in[a,b]
\right\},
\end{equation}
where \(u\colon [a,b]\to\mathbb R_{\geq 0}\) is continuous, smooth on
\((a,b)\), satisfies \(u(z)>0\) for \(z\in(a,b)\), and \(u(a)=u(b)=0\).
We define the \emph{south pole} and \emph{north pole} of \(S\) by
\[
P_S=(0,0,a),\qquad P_N=(0,0,b),
\]
respectively.

We impose the following regularity condition at the endpoints. The function
\[
v(z):=u(z)^2
\]
extends to a smooth function in a neighborhood of \([a,b]\), with
\[
v(a)=v(b)=0,\qquad v(z)>0 \text{ for } z\in(a,b),
\]
and
\[
v'(a)>0,\qquad v'(b)<0.
\]
Equivalently, \(v\) has simple zeros at the endpoints. This condition ensures
that \(S\) is smooth at the poles. Indeed, if
\[
F(x,y,z)=x^2+y^2-v(z),
\]
then \(S=F^{-1}(0)\) near the poles, and
\[
dF_{P_S}=-v'(a)\,dz\neq 0,\qquad
dF_{P_N}=-v'(b)\,dz\neq 0.
\]
Thus \(S\) is a smooth surface near \(P_S\) and \(P_N\) by the regular level
set theorem.

Moreover, since
\[
u(z)u'(z)=\frac12 v'(z)
\]
for \(z\in(a,b)\), we have
\[
\lim_{z\to a^+}u(z)u'(z)=\frac12v'(a)\neq 0,
\qquad
\lim_{z\to b^-}u(z)u'(z)=\frac12v'(b)\neq 0.
\]
This is the technical condition needed later for the ambient-coordinate
description of \(T^*S\) in equation~\eqref{eq: T*S} to extend correctly to
the poles. Finally, since
\[
u'(z)=\frac{v'(z)}{2\sqrt{v(z)}},
\]
we also have
\[
\lim_{z\to a^+}u'(z)=+\infty,
\qquad
\lim_{z\to b^-}u'(z)=-\infty.
\]
Thus the profile curve meets the \(z\)-axis orthogonally at the poles.

An \emph{equator} of \(S\) is any circle defined by
\[
C=\{(x,y,z)\in S\mid z=z_0\},
\]
where \(z_0\in(a,b)\) is a critical point of \(u\), i.e. \(u'(z_0)=0\).
Let \(D^*S\) denote the disk cotangent bundle of \(S\), namely the subset of
\(T^*S\) consisting of covectors with norm less than \(1\), where the norm is
induced by the Euclidean metric in the ambient space \(\mathbb R^3\).

\begin{theorem}\label{thm: ATF_surface_revolution}
Let $S\subset \mathbb{R}^3$ be a sphere of revolution with only one equator. Then there exists an almost toric fibration $f:(D^*S,\omega_{can})\to B_S$ where $B_S\subset \R\times \R_{\geq 0}$ with exactly two nodal singularities. 
\end{theorem}

The exact description of the base $B_S$ in Theorem \ref{thm: ATF_surface_revolution} is implicitly given within the proof, and we avoided stating it here to prevent cumbersome expressions. We consider in Proposition \ref{prop: ellipsoids} the case when $S$ is an ellipsoid of revolution; in this case, we will be more explicit with the description of the base $B_S$. 

\begin{rmk}
The construction in Theorem~\ref{thm: ATF_surface_revolution} is closely
related to the classical examples of integrable systems with an \(S^1\)-symmetry,
such as the quadratic spherical pendulum discussed in
\cite[Section~8]{cushman2002sign} and \cite[Chapter~4]{efstathiou2005metamorphoses}.
From the point of view of singularities, the resulting systems belong to the
semitoric class; see also \cite{LeFlochPalmer2024} for related examples of
semitoric families.

The additional feature needed in the present paper is that we work with disk
cotangent bundles, hence with symplectic domains with boundary. Therefore,
besides constructing the integrable system, we must choose the family of
closed curves in the reduced spaces so that the corresponding Lagrangian tori
remain inside the disk cotangent bundle. This produces a
almost-toric fibration adapted to the boundary and allows us to compute the
base region used later in the symplectic embedding applications.
\end{rmk}

Let $a, b, c > 0$. We denote by $\mathcal{E}(a, b, c) \subset \R^3$ the ellipsoid defined by the equation:
$$\frac{x^2}{a^2} + \frac{y^2}{b^2} + \frac{z^2}{c^2} = 1.$$
When the two parameters $a$ and $b$ coincide, we get an ellipsoid of revolution. Up to normalization, we can assume that $a = b = 1$. This is obtained by taking $u(z)=\sqrt{1-z^2/c^2}$ in \eqref{def: surf_rev}. In particular, notice that $S^2 = \mathcal{E}(1, 1, 1)$.

\begin{prop}\label{prop: ellipsoids}
For each $c > 0$, consider the almost toric fibration $f \colon (D^*\mathcal{E}(1,1,c),\, \omega_{\text{can}}) \to B_c$ given by Theorem~\ref{thm: ATF_surface_revolution}. The base region $B_c$ is bounded by the curve $(\mu,\, \mathcal{A}(\mu))$ defined in equation~\eqref{eq:bound_cbd_11c}.
\end{prop}
\begin{figure}[h!]
    \centering
    \begin{subfigure}[b]{0.32\textwidth}
        \centering
        \includegraphics[width=\textwidth]{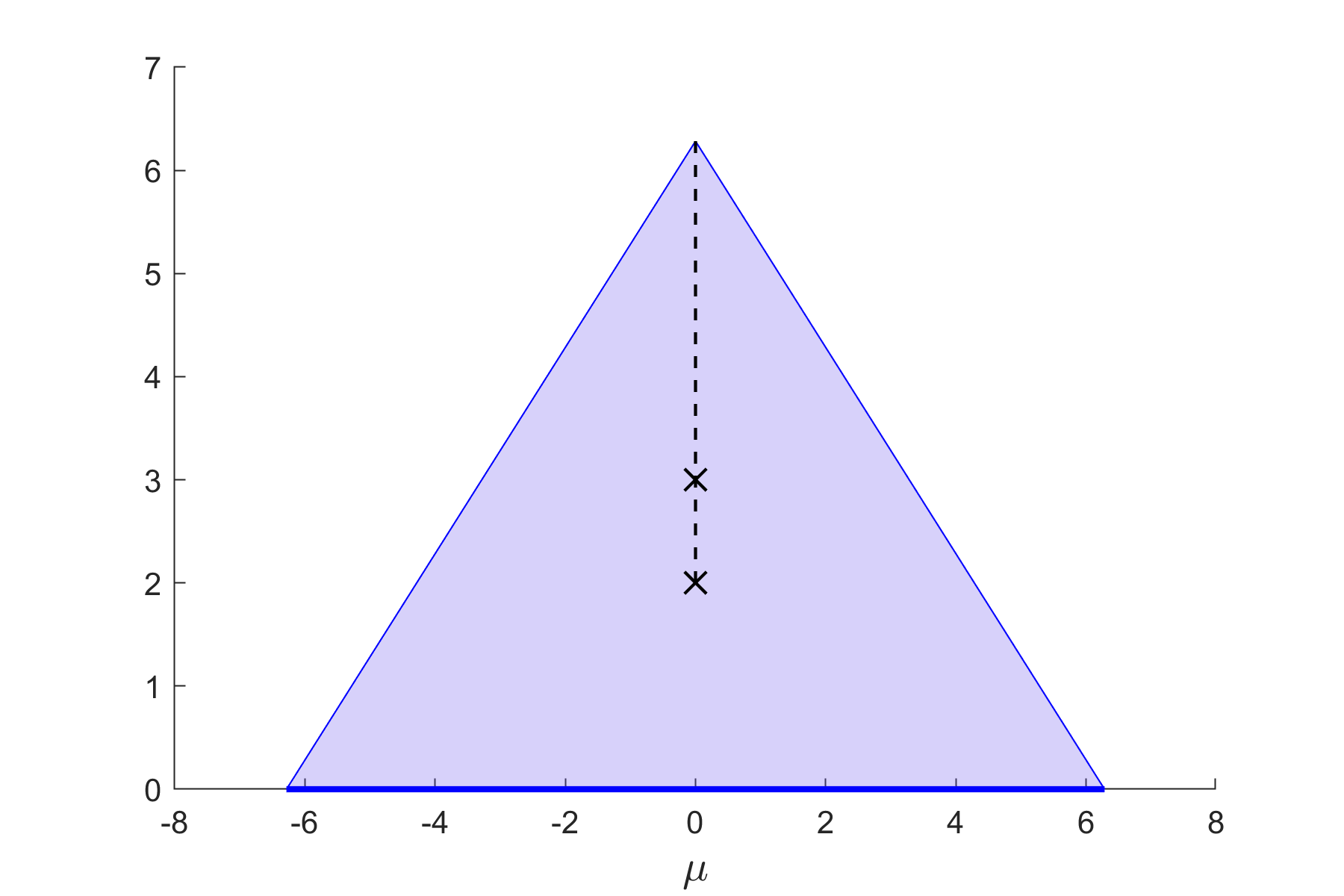}
        \caption{$c=1$}
        \label{fig:convex_111_2}
    \end{subfigure}
    \hfill
    \begin{subfigure}[b]{0.32\textwidth}
        \centering
        \includegraphics[width=\textwidth]{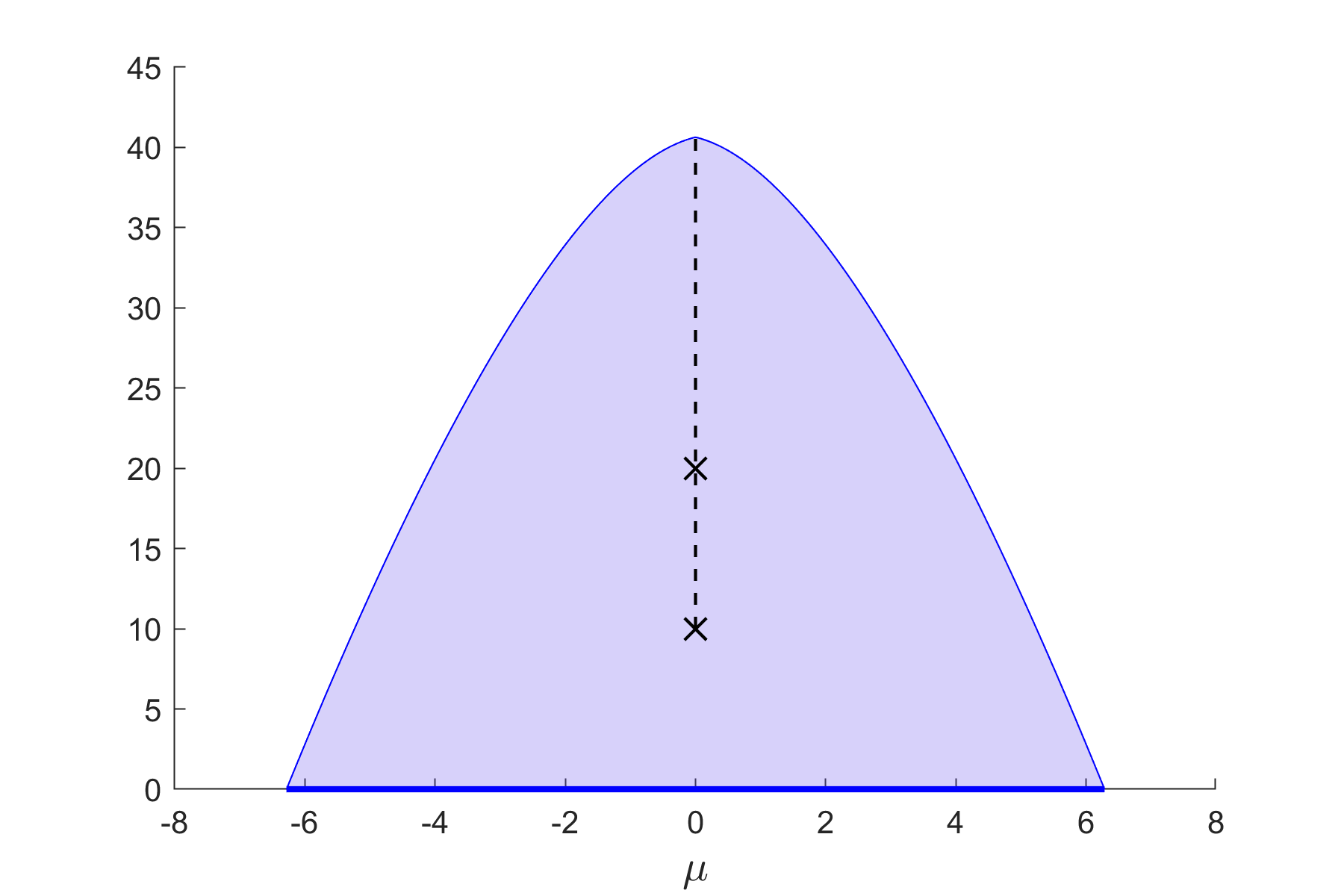}
        \caption{$c=10$}
        \label{fig:convex_112_2}
    \end{subfigure}
    \hfill
    \begin{subfigure}[b]{0.32\textwidth}
        \centering
        \includegraphics[width=\textwidth]{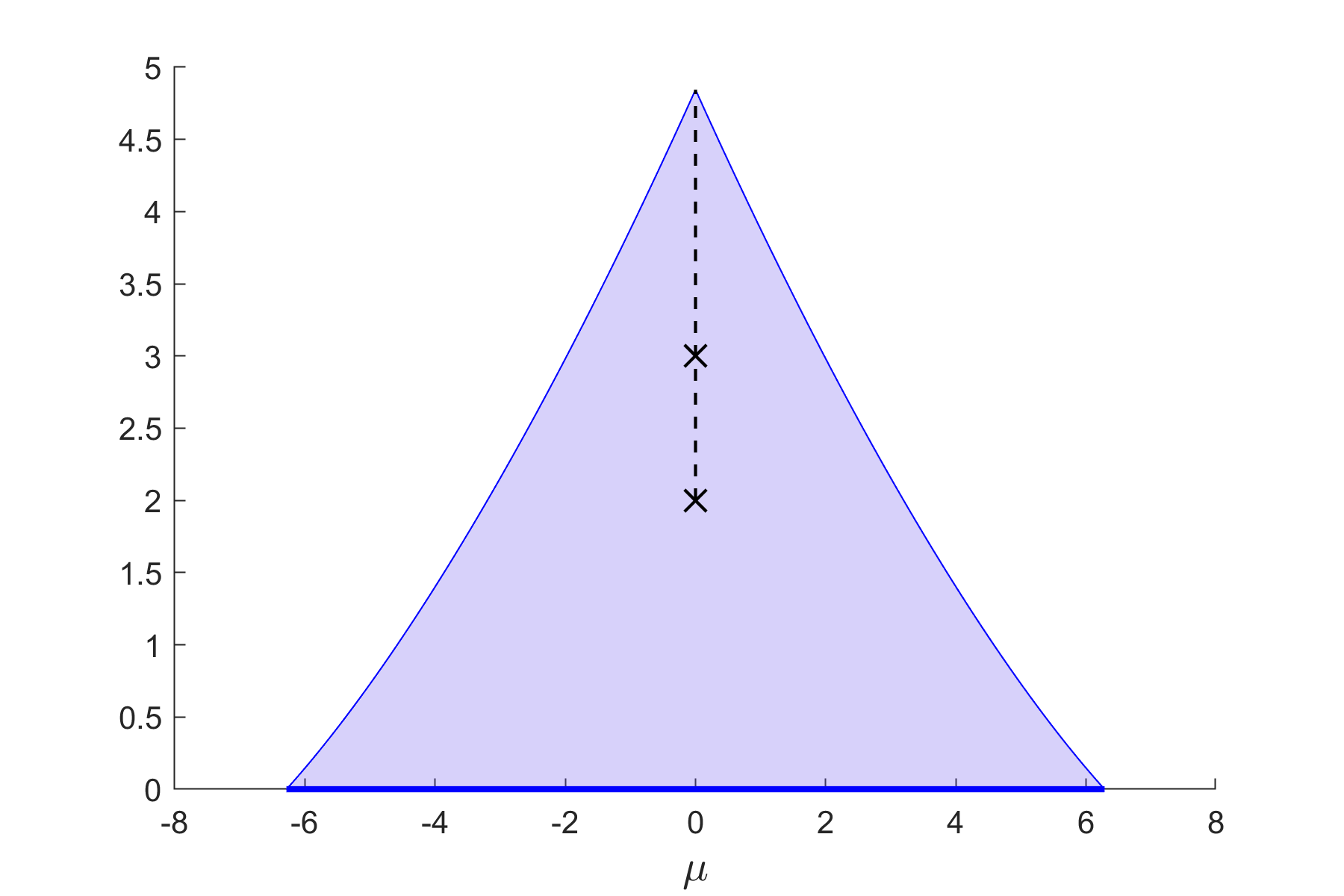}
        \caption{$c=0.001$}
        \label{fig:convex_11001}
    \end{subfigure}
    \caption{Base diagram for the almost toric fibration of Proposition \ref{prop: ellipsoids} on $D^*\mathcal{E}(1,1,c)$.}
    \label{fig:convex_base_11c_2}
\end{figure}

From the fibration obtained in Proposition~\ref{prop: ellipsoids}, we can
recover the toric domains obtained in \cite{Ferreira2023GromovWO} as follows.
The first step is a transferring-the-cut operation applied to the base diagrams
shown in Figure~\ref{fig:convex_base_11c_2}. This operation changes the
branch cut at the node: equivalently, we rotate the branch cut by \(180^\circ\)
so that it points in the opposite direction. The corresponding integral-affine change of chart is implemented by applying the matrix
\[
\begin{pmatrix}
1 & 0 \\
-1 & 1
\end{pmatrix}
\]
to the half-plane \(\mu\leq 0\). After this change of branch cut, we apply an
appropriate \(SL(2,\mathbb Z)\)-transformation to obtain the desired diagram;
see Figure~\ref{fig:convex_base_11c_2m}. One can perform an inverse nodal trade at the node with branch-cut emanating from the origin and a nodal slide at the second node, pushing the node all the way to the boundary. In the limit, we obtain a Lagrangian disk over the limiting point, which can be taken to be the cotangent fibre over one of the poles of the surface of revolution, agreeing with the case studied in \cite{Ferreira2023GromovWO}.

\begin{figure}[h!]
\centering
    \begin{tikzpicture}
        \node (left) at (0,0) {\includegraphics[scale=0.5]{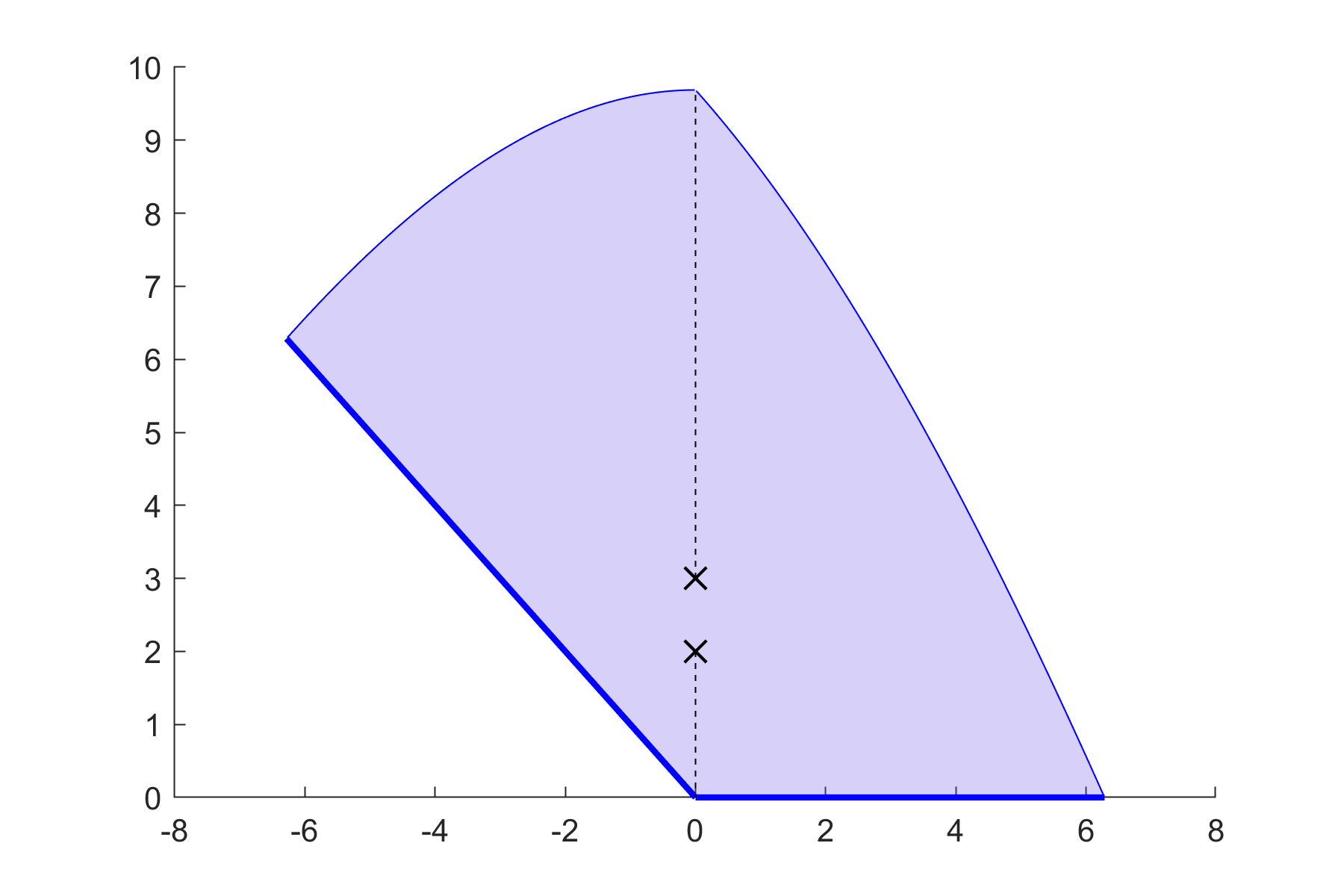}};
        \node (right) at (8.5,0) {\includegraphics[scale=0.5]{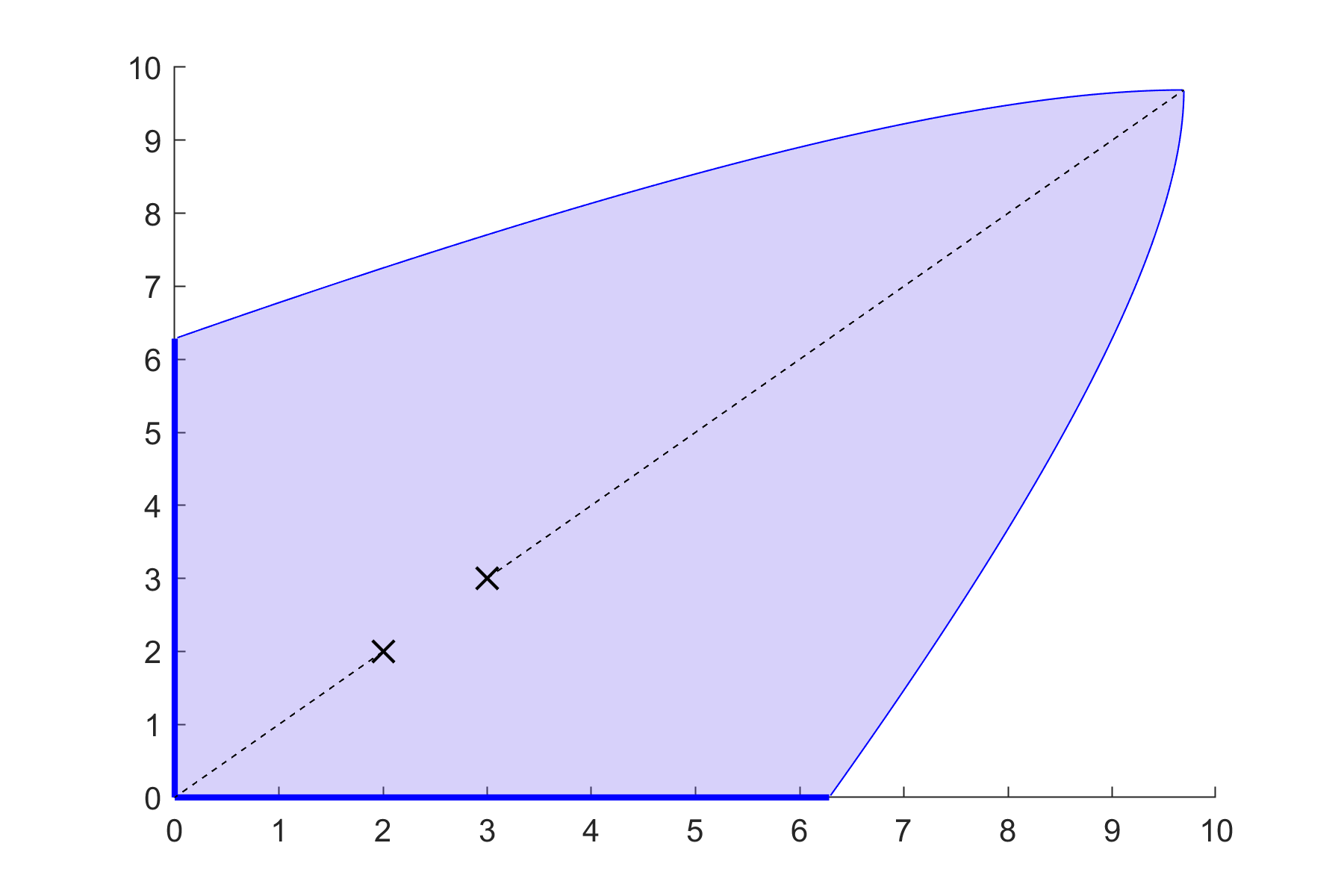}};
        
        \draw[-{Latex[scale=1.5]}, thick] (3,0) -- (5,0);

        \node at (4,0.7) {$\begin{pmatrix} 1 & 1 \\ 0 & 1 \end{pmatrix}$};
    \end{tikzpicture}
\caption{Base diagram for the almost toric fibration on $D^*\mathcal{E}(1,1,2)$ after a transferring-the-cut operation.}
\label{fig:convex_base_11c_2m}
\end{figure}

\begin{rmk}
Using the same argument, we can also obtain toric domains from the fibrations in Theorem \ref{thm: ATF_surface_revolution}. Furthermore, these toric domains agree with those obtained in \cite{Ferreira2023GromovWO}.
\end{rmk}

We can also consider the 3-dimensional ellipsoid as follows. Let $a,b,c,d>0$, we denote by $\mathcal{E}(a,b,c,d)\subset \R^4$ the surface defined by the equation:
$$\frac{x^2}{a^2}+\frac{y^2}{b^2}+\frac{z^2}{c^2}+\frac{w^2}{d^2}=1.$$

It turns out that using the same strategy, we can construct singular Lagrangian fibrations on the disk cotangent bundle of certain higher-dimensional ellipsoids, namely those with enough symmetries, as the next theorem shows.

\begin{theorem} \label{thm: highellips}
For each $c>0$, there exists a singular Lagrangian fibration $f:(D^*\mathcal{E}(1,1,c,c)),\omega_{can})\to A_c$, where the region $A_c$ is bounded by the surface $(\mu_1,\mu_2,\mathcal{A}(\mu_1,\mu_2))$, see equation \eqref{eq: boundary_11cc}.
\end{theorem}

\begin{figure}[h!]
    \centering
    \begin{subfigure}[b]{0.55\textwidth}
        \centering
        \includegraphics[width=\textwidth]{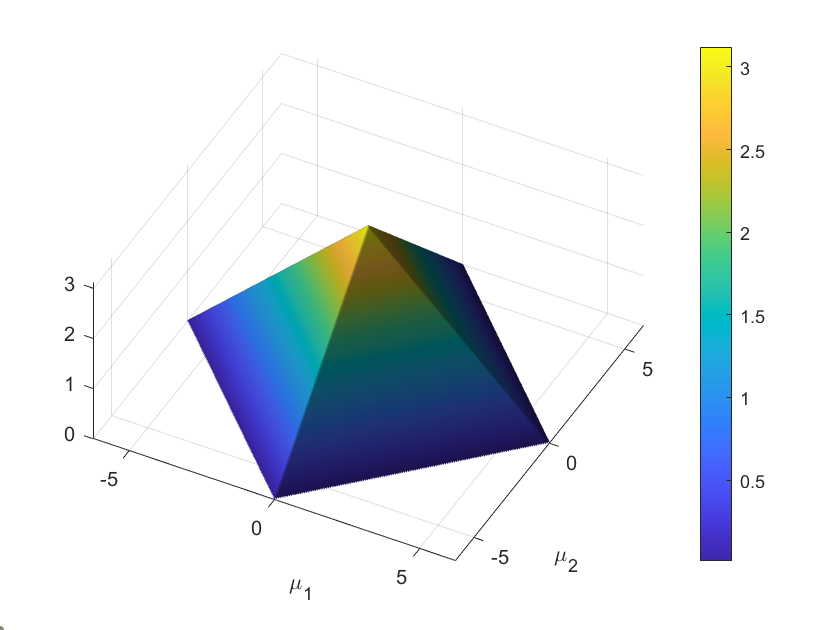}
        \caption{$c=1$}
        \label{fig:convex_1111}
    \end{subfigure}
    \hfill
    \begin{subfigure}[b]{0.55\textwidth}
        \centering
        \includegraphics[width=\textwidth]{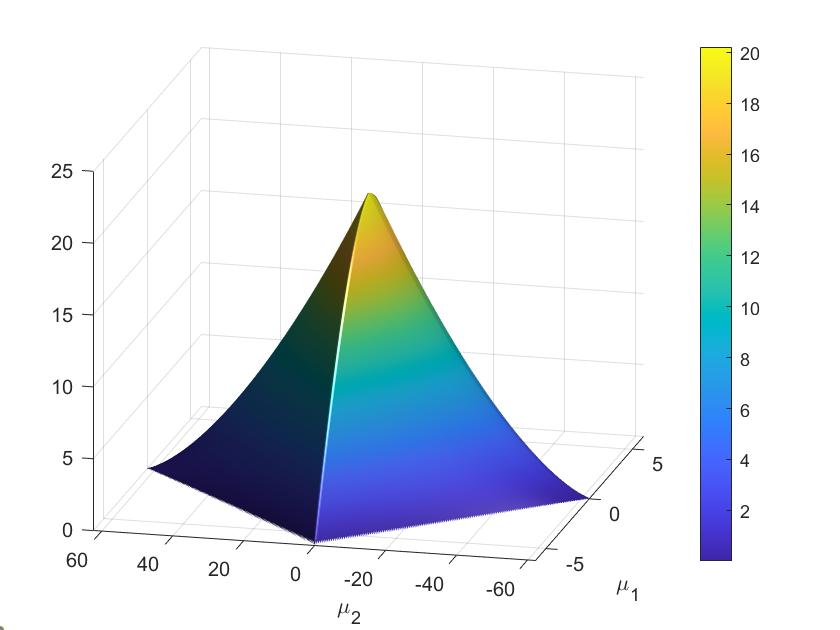}
        \caption{$c=10$}
        \label{fig:convex_1122}
    \end{subfigure}
    \hfill
    \begin{subfigure}[b]{0.55\textwidth}
        \centering
        \includegraphics[width=\textwidth]{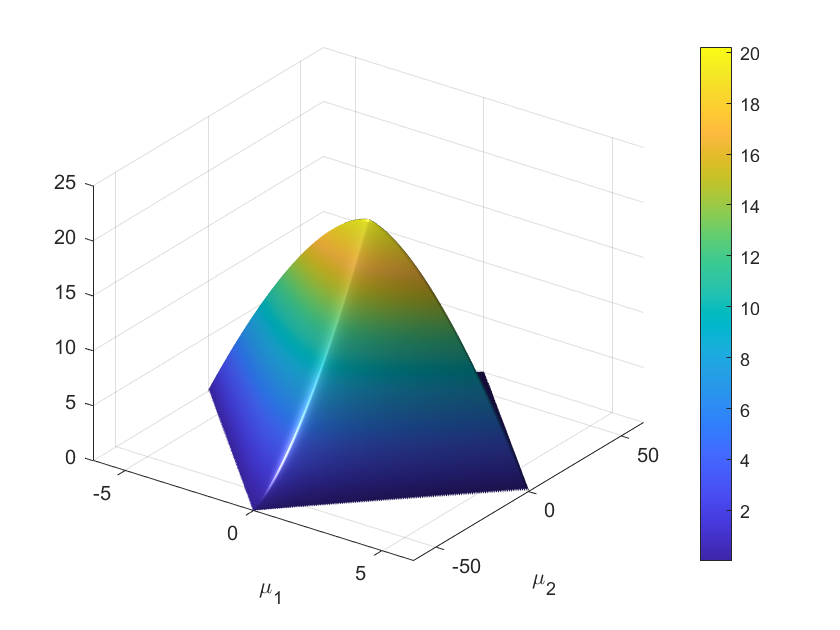}
        \caption{$c=10$}
        \label{fig:convex_1122_2}
    \end{subfigure}
    \caption{Base diagram for the singular Lagrangian fibration of Theorem \ref{thm: highellips} on $D^*\mathcal{E}(1,1,c,c)$. The color scale records the value of the
action coordinate \(\mathcal A\), equivalently the third coordinate of the
displayed base diagram.}
    \label{fig: convex_base_11cc}
\end{figure}

\begin{rmk}
Note that Figure \ref{fig: convex_base_11cc}(c) is a $90^\circ$ rotation of Figure \ref{fig: convex_base_11cc}(b).
\end{rmk}

Using the same strategy, we can also construct an almost toric fibration in the disk cotangent bundle $D^*D^2$ of the 2-dimensional disk $D^2$ of radius 1, also known as the Lagrangian bidisk. We adopt the same notation as in \cite{ramos2017symplectic} and define the Lagrangian bidisk as:
\[P_L=\{(p_1,p_2,q_1,q_2)\in \mathbb{R}^4~|~ p_1^2+p_2^2 \leq 1;~ q_1^2+q_2^2\leq 1 \},\]
with the symplectic form $\omega=\sum dp_j \wedge dq_j$.

\begin{theorem}\label{thm: lagbid}
There exists an almost toric fibration $f:(P_L,\omega)\to B_{P_L}$, where the region $B_{P_L}$ is bounded by the curve $(\mu,\mathcal{A}(\mu))$, see equation \eqref{eq: bound_lagbi}.
\end{theorem}

\begin{figure}[h!]
    \centering
    \includegraphics[width=0.7\textwidth]{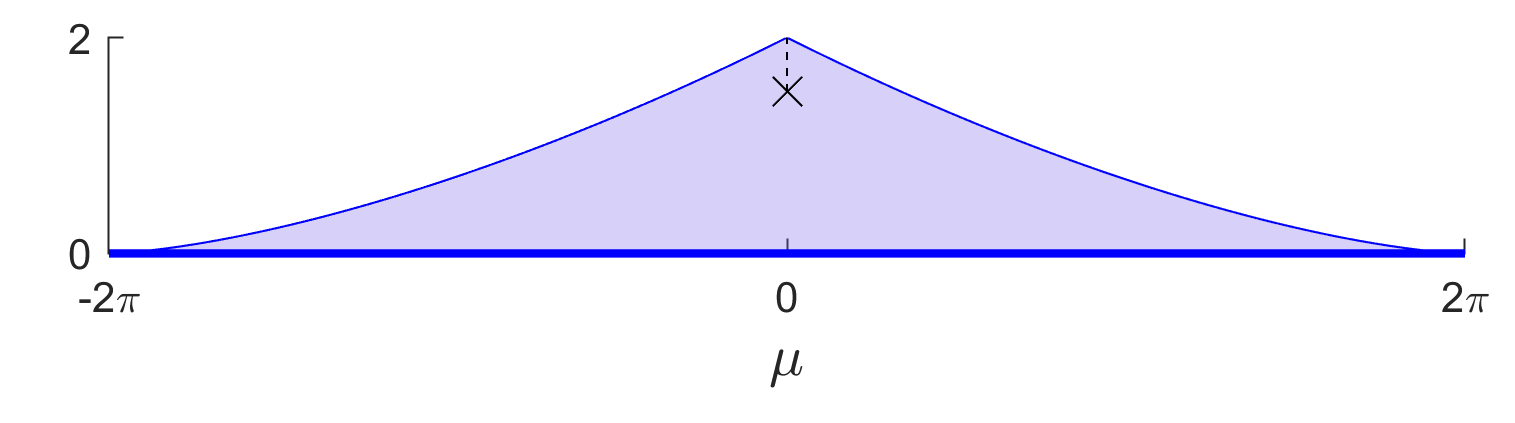} 
\caption{Base diagram for the almost toric fibration of Theorem \ref{thm: lagbid} on the Lagrangian Bidisk.}
\label{fig:convex_base_lagbi}
\end{figure}

\begin{rmk}\label{rmk: bidisk}
In a future project, we will use similar ideas to construct singular Lagrangian fibrations on higher-dimensional Lagrangian bidisks. The proof relies on constructing singular Lagrangian fibrations like those in this note, though the construction technique is more involved, since we will not start from the assumption of a $T^{n-1}$-Hamiltonian action.
\end{rmk}

From the fibration obtained in Theorem~\ref{thm: lagbid}, we can recover the
toric domain produced in \cite{ramos2017symplectic} as follows. First, we
perform a transferring-the-cut operation on the base diagram. This operation
changes the branch cut at the node; equivalently, we rotate the branch cut by
\(180^\circ\) so that it points in the opposite direction. The corresponding
integral-affine change of chart is implemented by applying the shear
\[
\begin{pmatrix}
	1 & 0 \\
	-1 & 1
\end{pmatrix}
\]
to the half-plane \(\mu \leq 0\). After this change of branch cut, we apply an
\(SL(2,\mathbb Z)\)-transformation to obtain the desired figure; see
Figure~\ref{fig:convex_base_lagbi_2}. We then apply a nodal trade at the
nodal singularity to obtain the toric domain.

\begin{figure}[h!]
\centering
\begin{tikzpicture}
    \node (left) at (0,0) {\includegraphics[width=0.5\linewidth,height=0.525\linewidth,keepaspectratio=false]{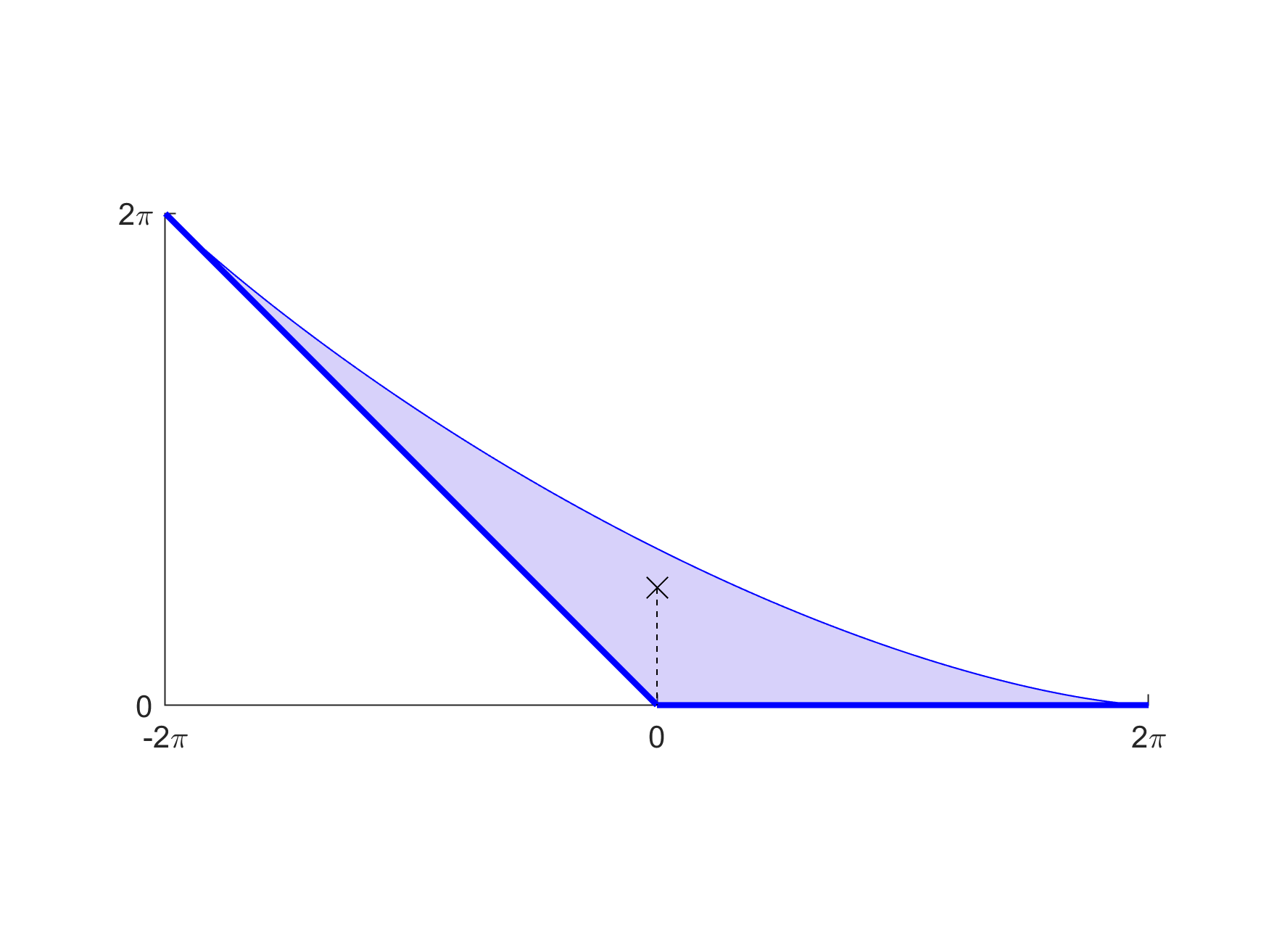}};
    
    \node (right) at (8,-0.05) {\includegraphics[scale=0.5]{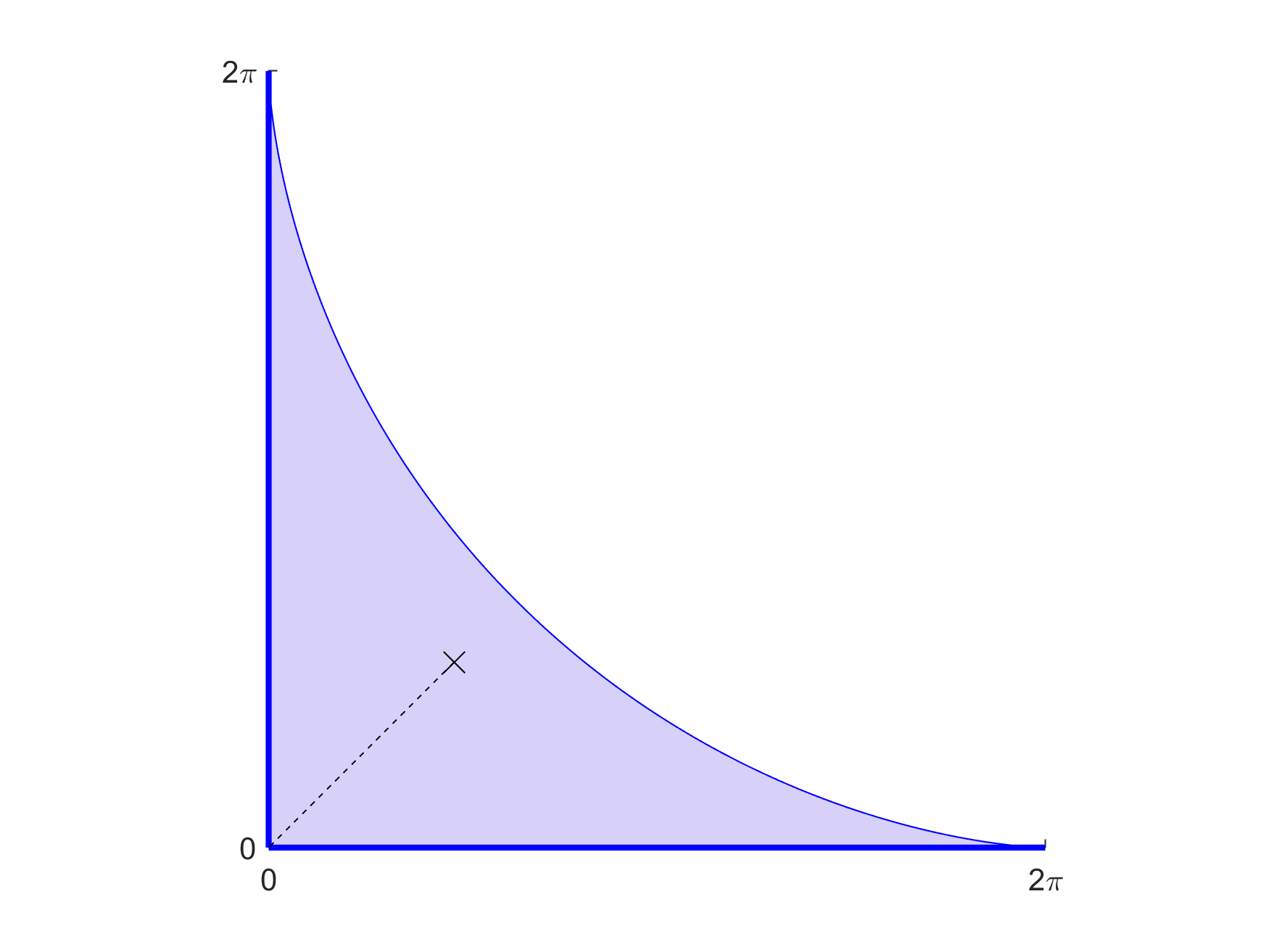}};
    
    \draw[-{Latex[scale=1.5]}, thick] (3,0) -- (5,0);

    \node at (4,0.7) {$\begin{pmatrix} 1 & 1 \\ 0 & 1 \end{pmatrix}$};
\end{tikzpicture}
\caption{Base diagram of the Lagrangian Bidisk after transferring the cut operation.}
\label{fig:convex_base_lagbi_2}
\end{figure}

\begin{rmk}
The almost toric fibrations presented in this section fall within the class of semi-toric almost toric fibrations presented in \cite{LeFlochPalmer2024}, for example.
\end{rmk}

\subsection{Embedding results} \label{sec: embres}

Ever since Gromov's groundbreaking result on the non-squeezing theorem \cite{Gromov1985PseudoHC}, many researchers have devoted extensive research to the study of symplectic embeddings. This is a fast-evolving area of work; for a relatively up-to-date account, see \cite{Schlenk2017SymplecticEP}. Gromov's result marked the beginning of the study of what is now known as the Gromov width of a symplectic manifold $(M,\omega)$, defined as the supremum of $a > 0$ such that there
exists a symplectic embedding $(B^{2n}(a), \omega_0) \to (M, \omega)$, where
$$B^{2n}(a)=\left\{(x_1,y_1,\ldots,x_n,y_n)\in \mathbb{R}^{2n}\mid \sum_{i=1}^n(x_i^2+y_i^2)\leq a/\pi\right\},$$
and $\omega_0=\sum_{i=1}^ndx_i\wedge dy_i$.

 The last decade has seen a surge of results on symplectic embedding problems using toric techniques, for example, \cite{ramos2017symplectic}, \cite{Ramos2017OnTR}, \cite{Ferreira2021SymplecticEI},\cite{Ferreira2023GromovWO}, and \cite{Ostrover2023FromLP}. One of the fundamental tools to study lower bounds for the Gromov width, via toric techniques, is the Traynor trick. This trick is basically about constructing a symplectic embedding of a ball into a Lagrangian product of a standard simplex and a cube, having the same volume of the ball (see \cite{Traynor1995SymplecticPC}, Lemma \ref{lem: traynor}). 

We now give a glimpse of how the singular Lagrangian fibrations constructed in the previous section can be used to study symplectic embedding problems into these spaces, in particular, studying the Gromov width.

From Proposition \ref{prop: ellipsoids} and the subsequent observations about obtaining toric domains out of these Lagrangian fibrations, we can see in particular that the symplectic balls embedded into $D^*\mathcal{E}(1,1,c)$ realizing the Gromov width, for $c\leq 1$ as in \cite{Ferreira2023GromovWO}, can also be seen within the picture of our fibration, by looking at the base in Figure \ref{fig:convex_base_11c_2}(c). More specifically, the ball realizing the Gromov width is represented in the
base diagram by the largest right isosceles triangle whose hypotenuse lies on
the horizontal axis and whose interior is contained in the region shown in Figure~\ref{fig:convex_base_11c_2}(c)
(see Figure~\ref{fig:triangle_width_c_leq_1}). Equivalently, this is the largest
isosceles triangle with horizontal base and with base length equal to twice
its height. After the transferring-the-cut operation and the subsequent
\(SL(2,\mathbb Z)\)-transformation, this triangle becomes an affine copy of
the standard simplex, and hence represents a symplectic ball.

\begin{figure}[h!]
    \centering
    \includegraphics[width=0.62\textwidth]{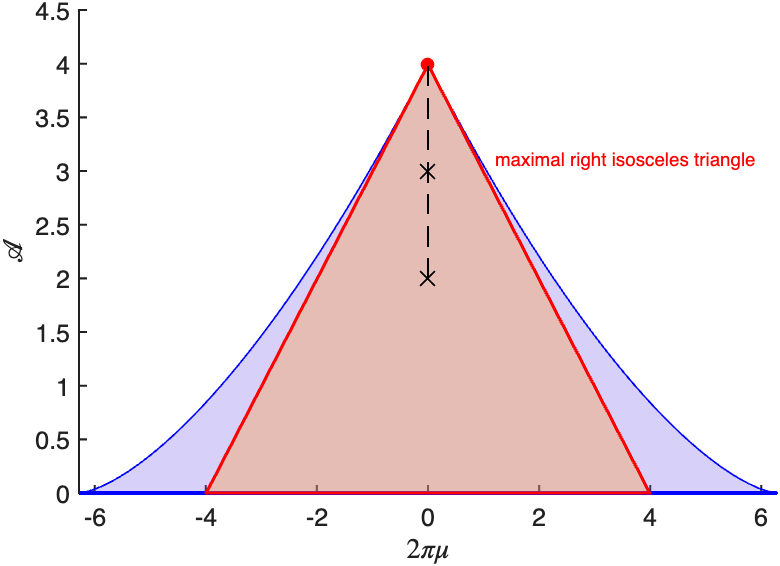}
    \caption{The right isosceles triangle representing the symplectic ball
    embedded in \(D^*\mathcal{E}(1,1,c)\) for \(c\leq 1\). The hypotenuse of the triangle lies on the horizontal axis, and its interior is contained in
    the base region of the almost toric fibration.}    \label{fig:triangle_width_c_leq_1}
\end{figure}

  For $c>1$, the existence of the symplectic ball realizing the Gromov width is proven using embedded contact homology, and it is not given by the inclusion on the toric domain obtained from the fibration in Figure \ref{fig:convex_base_11c_2}(b). This was already done in \cite{Ferreira2023GromovWO}, and as such, we will not return to it in these notes.

Using the fibration in Theorem \ref{thm: highellips} for $c=1$, we can show the following result.

\begin{prop}\label{prop: capacity_S^3}
The Gromov width of $D^*S^3$ is $2\pi$.
\end{prop}

We delay the proof of this result until Section \ref{sec: per}.

\begin{rmk}
We will prove in a future project that, as a matter of fact, the result in \ref{prop: capacity_S^3}, holds in every dimension. More specifically, $c_{Gr}(D^*S^n)=2\pi$ for every $n\geq 1$. The idea for this is part of the project mentioned in Remark \ref{rmk: bidisk}. The result in Proposition \ref{prop: capacity_S^3} and its generalization, mentioned here, are already known in \cite{Kislev2018BoundsOS}.
\end{rmk}

For $c\neq 1$ we can also estimate the Gromov width of $D^*\mathcal{E}(1,1,c,c)$.

\begin{theorem}\label{thm: Gromov_width}
The Gromov width of $D^*\mathcal{E}(1,1,c,c)$, for $c<1$, satisfies
\begin{equation*}
c_{Gr}(D^*\mathcal{E}(1,1,c,c))\geq
\begin{cases}
\begin{split}
2\pi c&, \textup{ if } c<1,\\
2\pi&, \textup{ if } c\geq 1.
\end{split}
\end{cases}
\end{equation*} 
\end{theorem}

The lower bound for $c_{Gr}(D^*\mathcal{E}(1,1,c,c))$ in Theorem \ref{thm: Gromov_width} can be computed in two different ways. On the one hand, we could use the fibration in Theorem \ref{thm: highellips} to find, with a little bit of work, as in the proof of Proposition \ref{prop: capacity_S^3}, an embedding of a ball of capacity $2\pi c$. On the other hand, and we choose this approach for its simplicity, we could use the fact that, for $c<1$, $D^*S^3$ symplectically embeds into the scaling $\frac{1}{\sqrt{c}}D^*\mathcal{E}(1,1,c,c)\subset T^*\mathbb{R}^4$ of $D^*\mathcal{E}(1,1,c,c)$, which implies by Proposition \ref{prop: capacity_S^3} and the monotonicity of symplectic capacities that, $2\pi=c_{Gr}(D^*S^3)\leq \frac{1}{c}c_{Gr}(D^*\mathcal{E}(1,1,c,c))$. Similarly, for $c>1$, we have that $D^*S^3$ symplectically embeds into $D^*\mathcal{E}(1,1,c,c)$, which implies that $2\pi=c_{Gr}(D^*S^3)\leq c_{Gr}(D^*\mathcal{E}(1,1,c,c))$.

\noindent{\textbf{Structure of the paper:}} In Section \ref{sec: genmet}, we state our definition of a singular Lagrangian fibration and introduce our general technique for their constructions. In Section \ref{sec: conslf}, we present the proofs for the results presented in Section \ref{sec: rescslf}. Finally, in Section \ref{sec: per}, we present the proofs for the results in Section \ref{sec: embres}.\\

\noindent{\textbf{Acknowledgements: }  Most of S. A-A.’s contribution to this work was carried out while
he was a postdoctoral researcher at the Department of Mathematics, Uppsala University,
and the final revisions were completed after he joined Universidad San Francisco de Quito
USFQ. Part of this work was carried out during a visit of A.V. to Uppsala University in August of 2024; he is very grateful to the Department of Mathematics for its hospitality and financial support. A.V thanks Johanna Bimmermann for making him aware of a result in \cite{Loi2013SymplecticCO} that was crucial in the proof of Proposition \ref{prop: capacity_S^3}.  S. A-A. was supported by the Knut and Alice Wallenberg Foundation through the grant KAW 2023.0294. A. V. was supported by the ISF Grant No. 2445/20. R. V. was supported by grant SEI E-26/200.230/2023 from Carlos Chagas Filho Research Support Foundation of the state of Rio de Janeiro (FAPERJ), grant \#2024/01351-6 from S\~{a}o Paulo Research Foundation (FAPESP), by ``Novos docentes'' grant from S\~{a}o Paulo University (USP).

\section{The General Method} \label{sec: genmet}
In this section, we recall the definition of Lagrangian fibrations as well as what we understand by singular Lagrangian fibration. Later, we describe a general method for constructing such kind of singular Lagrangian fibrations in our examples. 

\begin{definition}[{\cite[Definition~2.1]{symington71four}}]
	A locally trivial fibration of a symplectic manifold $(M^{2n},\omega)$ is a \textbf{regular Lagrangian fibration} if the fibers are smooth Lagrangians. A map $ \widetilde{\pi}:M^{2n} \to \widetilde{B}^n $ is a \textbf{Lagrangian fibration} if it restricts to a regular Lagrangian fibration over $ \widetilde{B} \setminus \widetilde{\Sigma} $, where $ \widetilde{B} \setminus \widetilde{\Sigma} $ is an open dense subset of $ \widetilde{B} $.
 \end{definition}

 \begin{rmk}
	We assume that our fibers are compact and connected. Hence, by Arnold-Liouville Theorem, the regular fibers are isomorphic to $ T^n $.
\end{rmk}

\begin{definition} \label{defsinglagfib}
A \textbf{singular Lagrangian fibration} of a symplectic manifold $(M^{2n}, \omega)$ is a Lagrangian fibration that is not regular.
\end{definition}

\begin{rmk}
	For instance, a toric fibration on a closed manifold is a singular Lagrangian fibration with singularities over the boundary of the moment polytope.
\end{rmk}

Our method to construct the singular Lagrangian fibrations obtained in Theorems \ref{thm: ATF_surface_revolution}, \ref{thm: highellips}, and \ref{thm: lagbid} is encapsulated in the following proposition.

\begin{prop}\label{generalmeth2}
Let \(M\) be an \(n\)-dimensional Riemannian manifold, and let \((T^{*}M, \omega)\) denote its cotangent bundle equipped with the canonical symplectic form. Assume the following:

\begin{enumerate}[label=\textnormal{(\roman*)}]
\item  There exists an effective Hamiltonian \(T^{n-1}\)-action on \(T^{*}M\), with moment map  
\(\mu: T^{*}M \rightarrow \mathbb{R}^{n-1}\), such that the action preserves the fibre-wise norm induced by the Riemannian metric.

\item For each \(\lambda \in \mathbb{R}^{n-1}\), the torus action on \(\mu^{-1}(\lambda)\) is free except at finitely many orbits. Moreover, the complement of the set of critical values of \(\mu\) is an open and dense subset of \(\mathbb{R}^{n-1}\).

\item For each \(\lambda\in\mathbb R^{n-1}\), there exists a continuous map
\[
f_\lambda:\mu^{-1}(\lambda)\to X_\lambda^2,
\]
where \(X_\lambda^2\) is a two-dimensional topological manifold embedded in
a fixed Euclidean space \(\mathbb R^m\), with \(m\) independent of \(\lambda\),
such that:
\begin{itemize}
\item the fibers of \(f_\lambda\) are connected and invariant under the
\(T^{n-1}\)-action;
\item the induced map
\[
\widetilde f_\lambda:\mu^{-1}(\lambda)/T^{n-1}\to X_\lambda^2
\]
is a homeomorphism;
\item on the subset of \(\mu^{-1}(\lambda)\) where the action is free,
\(\widetilde f_\lambda\) restricts to a diffeomorphism between smooth
manifolds.
\end{itemize}

\item 
For each \(\lambda \in \mathbb{R}^{n-1}\), the image  
$
f_{\lambda}\big(D^{*}M \cap \mu^{-1}(\lambda)\big)
$
is simply connected. Moreover, there exists a family of smooth simple closed
curves \(\{\gamma_{\lambda,s}\}_{s\in I}\) in \(X_\lambda^2\) whose regular
leaves foliate
$
f_{\lambda}\big(D^{*}M \cap \mu^{-1}(\lambda)\big)
$
up to the limiting leaves. The curves depend smoothly on \(\lambda\) and \(s\),
in the sense that, after choosing parametrizations, the map
\[
(\lambda,s,\theta)\longmapsto \gamma_{\lambda,s}(\theta)
\]
is smooth as a map to the fixed ambient space \(\mathbb R^m\).
\end{enumerate}

Then the unit cotangent disc bundle \(D^*M\subset T^*M\) admits a singular
Lagrangian torus fibration whose regular fibers are
\[
L_{\lambda,s}
=
f_\lambda^{-1}(\gamma_{\lambda,s})\cap D^*M.
\]
Equivalently, the fibration is obtained by combining the moment map coordinates
\(\mu\) with the leaf parameter of the chosen foliation of the reduced spaces.
When the leaves are parametrized by symplectic area, as in the examples below,
the fibration is written on the regular part as
\[
(\mu,\mathcal A):D^*M\to \mathbb R^{n-1}\times\mathbb R_{\geq 0}.
\]
\end{prop}

\begin{proof}
We adapt the construction method from~\cite{gross2001examples} to produce a singular Lagrangian torus fibration on \(D^*M\).

By assumption~(i), the Hamiltonian \(T^{n-1}\)-action preserves the fibre-wise norm on \(T^*M\), and thus restricts to a well-defined action on the unit cotangent disc bundle \(D^*M\).  
Assumption~(ii) guarantees that for each fixed \(\lambda \in \mathbb{R}^{n-1}\), the level set \(\mu^{-1}(\lambda)\) is smooth except at finitely many orbits where the action is not free. At regular values of \(\mu\), the reduced space \(\mu^{-1}(\lambda)/T^{n-1}\) is a smooth 2-dimensional manifold.

From assumption~(iii), we are given a continuous map  
\(f_\lambda : \mu^{-1}(\lambda) \to X^2_\lambda \subset \mathbb{R}^m\),  
whose fibers are connected and invariant under the \(T^{n-1}\)-action, and whose induced map  
\(\widetilde{f}_\lambda : \mu^{-1}(\lambda)/T^{n-1} \to X^2_\lambda\)  
is a homeomorphism. On the open dense subset where the action is free, the quotient is smooth and \(\widetilde{f}_\lambda\) restricts to a diffeomorphism. This allows us to transfer local geometric data between \(X^2_\lambda\) and \(\mu^{-1}(\lambda)/T^{n-1}\).

Assumption~(iv) provides a smooth family of simple closed curves \(\{\gamma_{\lambda,s}\}_{s \in I}\) in \(X^2_\lambda\), smoothly parametrized by both \(\lambda\) and \(s\), and foliating the image \(f_\lambda(D^*M \cap \mu^{-1}(\lambda))\). For each pair \((\lambda,s)\), define
\[
L_{\lambda,s}
:=
f_\lambda^{-1}(\gamma_{\lambda,s})\cap D^*M.
\]

In the regular case, where the action is free, the lift of \(\gamma_{\lambda,s}\) under \(\widetilde{f}_\lambda^{-1}\) defines a smooth embedded circle in the reduced space, and its preimage under the quotient map defines a smooth embedded torus
\[
L_{\lambda, s} \cong T^{n-1} \times S^1 \cong T^n.
\]

To prove that \(L_{\lambda, s}\) is Lagrangian, we closely follow the argument in \cite[Theorem 1.2]{gross2001examples}. For a noncritical point \(x \in \mu^{-1}(\lambda)\), we have
\[
T_x(\mu^{-1}(\lambda)/T^{n-1}) = \frac{T_x \mu^{-1}(\lambda)}{(T_x \mu^{-1}(\lambda))^{\omega}},
\]
where \((T_x \mu^{-1}(\lambda))^{\omega}\) is the symplectic orthogonal, which coincides with the tangent space to the orbit \(T^{n-1} \cdot x\), generated by the vector fields \(X_1, \dots, X_{n-1}\). Since \(\mu^{-1}(\lambda)/T^{n-1}\) is a 2-dimensional symplectic manifold, any 1-dimensional submanifold is Lagrangian. Thus, \(\widetilde{f}_\lambda^{-1}(\gamma_{\lambda, s})\) is a simple closed Lagrangian submanifold in the reduced space, and \(L_{\lambda, s}\) is its preimage in \(\mu^{-1}(\lambda)\).

Let \(Y\) be a lift of a generator of the tangent space to \(\widetilde{f}_\lambda^{-1}(\gamma_{\lambda, s})\) at \(x \mod T^{n-1}\). At regular points, the tangent space \(T_x L_{\lambda, s}\) is spanned by \(X_1, \dots, X_{n-1}, Y\). Since \(\omega(X_i, X_j) = 0\) and \(Y \in T_x \mu^{-1}(\lambda)\), it follows that \(\omega(X_i, Y) = 0\) because \(X_i \in (T_x \mu^{-1}(\lambda))^\omega\). Therefore, \(L_{\lambda, s}\) is a smooth Lagrangian submanifold wherever it avoids the singular set of \(\mu\).

 When \(L_{\lambda,s}\) meets the non-free locus of the torus action, or when the corresponding curve is a limiting leaf of the chosen foliation, the fibre
may fail to be a smooth torus. These are precisely the singular fibers of the
constructed fibration.  We can define our Lagrangian fibration by combining the moment map coordinates
\(\mu\) with the leaf parameter of the chosen foliation of the reduced spaces. Since the free regular locus is open and dense by the
assumptions, the above construction gives a regular Lagrangian torus fibration over an
open dense subset of the base, and hence a singular Lagrangian torus fibration
on \(D^*M\).

This concludes the proof.

\end{proof}

\begin{rmk}
In the examples considered in this paper, all the assumptions of Proposition~\ref{generalmeth2} are satisfied. More specifically:

\begin{itemize}
    \item The torus action on \(T^*M\) is Hamiltonian and preserves the fibre-wise norm, satisfying condition~\textnormal{(i)}.
    
    \item The image of the singular
locus in the base is contained in a finite collection of codimension-one Whitney
stratified cuts. After removing these cuts, the action coordinates give the
standard integral-affine structure on the base. Therefore, the resulting
fibrations admit convex base diagrams in the sense of
\cite[Definition~2.33]{achigandrango2024singular}.

    \item The maps \(f_\lambda : \mu^{-1}(\lambda) \to X^2_\lambda\) are explicitly constructed so that the induced maps \(\widetilde{f}_\lambda\) are homeomorphisms, and restrict to diffeomorphisms on the regular part of the quotient, as stated in condition~\textnormal{(iii)}.

    \item In some examples, an initial family of closed curves arises from level sets of
a global function depending smoothly on both \(\lambda\) and an auxiliary
parameter. This gives an auxiliary Lagrangian fibration of the form
\((\mu,H)\). However, this is not always the restricted almost-toric fibration
used later in the paper; in such cases, the foliation is modified in order to
obtain the desired boundary-adapted base diagram, see, for example
Figure~\ref{fig:folX20}.
\end{itemize}
\end{rmk}

In all the examples in this paper, we will work with restricted almost-toric fibrations as defined in \cite[Definition 2.35]{achigandrango2024singular}. We represent these Lagrangian fibrations using a base diagram (see \cite[Section 2]{achigandrango2024singular}).

We now show how this idea works in a classical toy model, taken from \cite{auroux2007mirror}.

\begin{ex}[Toy model]\label{ex: toy model}
For $\varepsilon\neq 0$, consider the locus $F:=V(z_1z_2-z_3(z_3-\varepsilon))\subset \mathbb{C}^3$. There exists a Hamiltonian $S^1$-action on $F$ given by 
$$e^{i\theta}\cdot(z_1,z_2,z_3)=(e^{i\theta}z_1,e^{-i\theta}z_2,z_3).$$
This action has moment map $\mu(z_1,z_2,z_3)=\frac{1}{2}(|z_1|^2-|z_2|^2)$. Let $\pi: F\rightarrow \mathbb{C}$ be the singular symplectic fibration given by $\pi(z_1,z_2,z_3)=z_3$. Notice that for every $z\in \mathbb{C}$ with $z\neq0, \varepsilon$ we have that $\pi^{-1}(z)\cong \mathbb{C^*}$ whereas $\pi^{-1}(0)$ and $\pi^{-1}(\varepsilon)$ can be identified with the cone $\{(z_1,z_2)\in\mathbb{C}^2\mid z_1z_2=0\}$, these are the singular fibers. See Figure \ref{fig: toy_model}.

\begin{figure}[h!]
\centering
\includegraphics[scale=0.6]{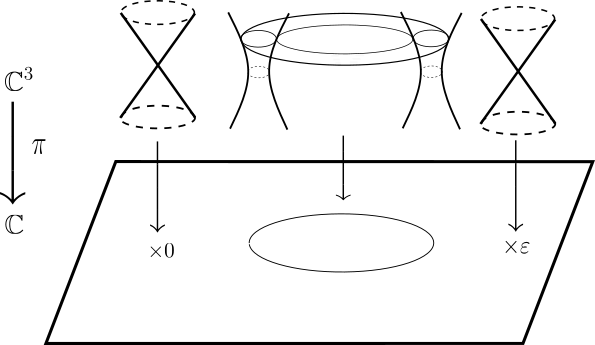}
\captionsetup{justification=centering}
\caption{The fibration $\pi$ in Example \ref{ex: toy model}.}
\label{fig: toy_model}
\end{figure}

Consider now the map $f:\mathbb{C}\rightarrow \mathbb{R}$ given by $f(z)=|z-a|^2$ for $a\in \mathbb{R}$ and $a\neq 0,\varepsilon$. Then, it is not hard to see that $(\mu,f\circ\pi)$ is an integrable system. From the computation
\begin{align*}
d\mu &=z_1d\bar{z}_1+\bar{z}_1dz_1-z_2d\bar{z}_2-\bar{z}_2dz_2,\\
d(f\circ \pi)&=(\bar{z}_3-a)dz+(z_3-a)d\bar{z}_3,
\end{align*} 
we can see that the system $(\mu,f\circ\pi)$ has an elliptic singularity at every point $(z_1,z_2,a)\in F$, whereas the points $(0,0,0)$ and $(0,0,\varepsilon)$ are also singularities, however of nodal kind. To see this is enough to notice that the subspace generated by the Hessians of $\mu$ and $(f\circ \pi)$ agree with that of the Hamiltonian system\footnote{This is the local model defining nodal or focus-focus singularities.} $(H_1,H_2):\mathbb{C}^2\rightarrow \mathbb{R}$ given by $(H_1,H_2)(z_1,z_2)=\big(\frac{i}{2}\mathfrak{Im}(z_1^2+z_2^2),\mathfrak{Im}(z_1\bar{z}_2)\big)$. Using the idea in Proposition \ref{generalmeth2}, we can construct a standard Lagrangian fibration on $F$. It is not hard to see that this fibration has the same fibers as the continuous map $(\mu,A):F\rightarrow \mathbb{R}\times \mathbb{R}_{\geq 0}$, where $A(z_1,z_2,z_3)$ is the area of a disk in $F$ with boundary in the Lagrangian torus $(\mu,f\circ\pi)^{-1}(\frac{1}{2}(|z_1|^2-|z_2|^2),|z_3-a|^2)$ and projecting to the curve $|z-a|^2=|z_3-a|^2$ in the base $\mathbb{C}$ of the fibration. The image of the map $(\mu, A)$ represents an almost toric base diagram of the quadric $T^*S^2$, which includes those depicted in Figure \ref{fig:convex_base_11c_2}.
\end{ex}

\section{Constructing Singular Lagrangian Fibrations}\label{sec: conslf}

In this section, we apply the general method to some examples.

\subsection{Disk cotangent bundle of ellipsoids of revolution}

\begin{proof}[Proof of Theorem \ref{thm: ATF_surface_revolution}]
Let $u: [a, b] \to \mathbb{R}_{\geq 0}$ be the profile curve for the surface of revolution $S$. Without loss of generality, we assume $a < z_0 = 0 < b$, where $z_0$ is the unique critical point of $u$. Considering \( T^*S \) with its standard inclusion in \( \mathbb{R}^6 \), we have:
\begin{equation}\label{eq: T*S}
T^*S = \left\{ (\xi, \eta) \in \mathbb{R}^3 \times \mathbb{R}^3 \mid \xi_1^2 + \xi_2^2 - (u(\xi_3))^2 = 0, \ \xi_1 \eta_1 + \xi_2 \eta_2 - u(\xi_3) u'(\xi_3) \eta_3 = 0 \right\}.
\end{equation}
In these coordinates, the canonical symplectic form on \( T^*S \) coincides with the restriction of the standard symplectic form on \( \mathbb{R}^6 \), given by \( \omega = \sum d\xi_j \wedge d\eta_j \).

Now, consider the natural Hamiltonian \( S^1 \)-action on \( T^*S \) given by:
\[
e^{i\theta} \cdot \begin{pmatrix}
	\xi_1 \\
	\xi_2 \\
	\xi_3 \\
	\eta_1 \\
	\eta_2 \\
	\eta_3
\end{pmatrix} \longmapsto \begin{pmatrix}
	\cos 2\pi\theta & -\sin 2\pi\theta & 0 & 0 & 0 & 0 \\
	\sin 2\pi\theta & \cos 2\pi\theta & 0 & 0 & 0 & 0 \\
	0 & 0 & 1 & 0 & 0 & 0 \\
	0 & 0 & 0 & \cos 2\pi\theta & -\sin 2\pi\theta & 0 \\
	0 & 0 & 0 & \sin 2\pi\theta & \cos 2\pi\theta & 0 \\
	0 & 0 & 0 & 0 & 0 & 1
\end{pmatrix} \begin{pmatrix}
	\xi_1 \\
	\xi_2 \\
	\xi_3 \\
	\eta_1 \\
	\eta_2 \\
	\eta_3
\end{pmatrix}.
\]
This action has moment map \( \mu(\xi, \eta) = 2\pi(\xi_1 \eta_2 - \xi_2 \eta_1) \), and preserves the norm of the cotangent fibers $|\eta|$. We consider the map \( f: T^*S \to \mathbb{R}^3 \) given by:
\[
f(\xi, \eta) := \left( \xi_3,  \eta_3, |\eta|^2 \right).
\]
And we define $f_\lambda:=f|_{\mu^{-1}(\lambda)}$ and $X^2_{\lambda}:=f(\mu^{-1}(\lambda))$.

To use Proposition \ref{generalmeth2}, we need to prove that $f_\lambda$ satisfies the desired hypothesis. It is easy to see that the fibers of $f_\lambda$ are connected and preserved by the $S^1$ action. Now, we are going to check that the induced map $\widetilde{f}_{\lambda}:\mu^{-1}(\lambda)/S^1 \to X^2_{\lambda}$ is a bijection for every $\lambda \in \mathbb{R}$, and that $f_{\lambda}(D^*S \cap \mu^{-1}(\lambda))$ is simply connected for every $\lambda$. We also provide a foliation of this set given by the norm of the cotangent fibers. By abuse of notation, we refer to a fixed value of the moment map and a fixed value of the norm of the cotangent fibers as $\mu$ and $|\eta|$, respectively.

\begin{clm}
\begin{enumerate}[label=(\roman*)]
    \item \label{claim: i} The following relation holds:
    \begin{equation}\label{eq: relation}
       \eta_3^2 u(\xi_3)^2 \left(1 + u'(\xi_3)^2\right) = u(\xi_3)^2|\eta|^2 -\frac{\mu^2}{4\pi^2} .
    \end{equation}
    \item \label{claim: ii} For fixed $\mu$ and $|\eta| \neq 0$, the locus in the plane $(\xi_3, \eta_3)$ defined by equation \eqref{eq: relation} is a closed curve or a point. For $|\eta| = 0$, we have $\mu = 0$, and the locus is $[a, b]\times \{0\}$, i.e., the interval $[a, b]$ on the $\xi_3$ axis.
    \item \label{claim: iii} $f(D^*S \cap \mu^{-1}(\lambda))$ is simply connected and has a foliation induced by the norm of the cotangent fibers.
    \item \label{claim: iv} The induced map $\widetilde{f}_{\lambda}:\mu^{-1}(\lambda)/S^1 \to X^2_{\lambda}\subset \mathbb{R}^3$ is a bijection for every $\lambda \in \mathbb{R}$
\end{enumerate}
\end{clm}

\begin{proof}[Proof of the Claim]

\noindent\textbf{\ref{claim: i}.} \\
We begin by considering the following equations:
\begin{align}
    \frac{\mu^2}{4\pi^2} &= \xi_1^2 \eta_2^2 - 2\xi_1 \xi_2 \eta_1 \eta_2 + \xi_2^2 \eta_1^2, \label{eq:2mu} \\
    u(\xi_3)^2 u'(\xi_3)^2 \eta_3^2 &= \xi_1^2 \eta_1^2 + 2\xi_1 \xi_2 \eta_1 \eta_2 + \xi_2^2 \eta_2^2, \label{eq:2xi} \\
    |\eta|^2 - \eta_3^2 &= \eta_1^2 + \eta_2^2, \label{eq:2normeta} \\
    u(\xi_3)^2 &= \xi_1^2 + \xi_2^2. \label{eq:2xi3}
\end{align}
Equation \eqref{eq:2mu} is obtained by squaring the moment map $\mu$, while \eqref{eq:2xi} results from squaring the cotangent condition in the definition of $T^*S$, see equation \eqref{eq: T*S}. Note that adding the right-hand sides of \eqref{eq:2mu} and \eqref{eq:2xi} yields the same expression as multiplying the right-hand sides of \eqref{eq:2normeta} and \eqref{eq:2xi3}. This algebraic identity implies the desired relation.

\vspace{1em}
\noindent\textbf{\ref{claim: ii}.} \\
We analyze the relation
\[
u(\xi_3)^2 |\eta|^2 - \frac{\mu^2}{4\pi^2} \geq 0,
\]
which must hold for the curve defined by \eqref{eq: relation}, with fixed $\mu$ and $|\eta|$, to be nonempty. For fixed $\mu$ and $|\eta| \neq 0$, the left-hand side is non-negative over an interval of $\xi_3$. Since $u(0)$ is the maximal value of $u$, it follows that
\[
\mu^2 \leq 4\pi^2 u(0)^2 |\eta|^2.
\]

If the inequality is strict, i.e., $\mu^2 < 4\pi^2 u(0)^2 |\eta|^2$, then by the properties of $u$, there exist exactly two points
\[
a \leq \xi_{3,\min}^{\eta,\mu} < 0 < \xi_{3,\max}^{\eta,\mu} \leq b
\]
satisfying
\[
4\pi^2 u(\xi_3)^2 |\eta|^2 - \mu^2 = 0.
\]
Thus, equation \eqref{eq: relation} defines a closed curve in the plane $(\eta_3,\xi_3)$, for $\xi_3 \in [\xi_{3,\min}^{\eta,\mu}, \xi_{3,\max}^{\eta,\mu}]$.

If $\mu^2 = 4\pi^2 u(0)^2 |\eta|^2$, then the equality $4\pi^2 u(\xi_3)^2 |\eta|^2 - \mu^2 = 0$ holds only at $\xi_3 = 0$, and hence the solution is a single point.

In the case $|\eta| = 0$, the inequality forces $\mu = 0$, and the equation \eqref{eq: relation} reduces to either $\eta_3 = 0$, with $\xi_3$ arbitrary in $[a, b]$, or $u(\xi_3)=0$, which only happens for $\xi_3=a,b$ and therefore $\eta_3=0$ in that case, by \eqref{eq: T*S}.

\vspace{1em}
\noindent\textbf{ \ref{claim: iii}.} \\
This follows directly from Claim \ref{claim: ii}. The curves described above are in the level sets defined by fixing $|\eta|$, and thus are mapped by $f$ to closed curves $\gamma_{\mu,|\eta|}$. As $|\eta| \to 0$, these curves shrink: they collapse to a point when $\mu \neq 0$ or to a line segment when $\mu = 0$.

\vspace{1em}
\noindent\textbf{\ref{claim: iv}.} \\
We aim to prove that distinct $S^1$-orbits map to distinct points under $\widetilde{f}_\lambda$. Suppose two orbits map to the same image. Then they must share the same values of $(\xi_3, \eta_3, |\eta|^2)$. To prove injectivity, we show that for each such triple, there exists a unique orbit under the $S^1$-action.

Using the $S^1$-action, we may assume without loss of generality that either $\xi_1 = 0$ or $\eta_1 = 0$, which simplifies the analysis of the equations.

\textit{Case 1:} $u(\xi_3) \neq 0$. Without loss of generality, we can assume $\xi_1 = 0$, by moving along an orbit of the $S^1$-action. Then the equations yield:
\[
\xi_2 = \pm u(\xi_3), \qquad 
\eta_1 = -\frac{\lambda}{2\pi \xi_2}, \qquad 
\eta_2 = \frac{u(\xi_3) u'(\xi_3) \eta_3}{\xi_2}.
\]
These give two points in the same $S^1$-orbit.

\textit{Case 2:} $u(\xi_3) = 0$. We now set $\eta_1 = 0$. Then we obtain:
\[
\xi_1 = \xi_2 = 0, \qquad 
\eta_2 = \pm \sqrt{|\eta|^2}, \qquad 
\eta_3=0.
\]
Again, we get exactly two points related by a rotation of $\pi$, regarding the $S^1$-action.

In both cases, we conclude that the triple $(\xi_3, \eta_3, |\eta|^2)$ uniquely determines an $S^1$-orbit. This completes the proof of injectivity.
\end{proof}

To construct an almost toric fibration on \(D^*S\), we begin by choosing an
appropriate foliation of the base of each \(f_\lambda\). A natural choice
arises from the norm of the cotangent fibers. For \(\mu=0\), the level
\(|\eta|=0\) is the image of the zero section in \(X_0^2\); explicitly, it is
the segment
\[
Z_S=\{(\xi_3,\eta_3,|\eta|^2)\mid \xi_3\in[a,b],\ \eta_3=0,\ |\eta|^2=0\}.
\]
However, in the resulting almost toric base diagram, whose coordinates are
\((\mu,\mathcal A)\), this whole segment is mapped to the single point
\((0,0)\), since \(\mu=0\) on the zero section and the area coordinate
\(\mathcal A\) vanishes there. In Figure~\ref{fig:folX20}(a), we indicate the segment explicitly. In the particular ellipsoid examples shown in
Figure~\ref{fig: convex_base_11c}, its image is the marked point on the
horizontal axis, labelled according to the corresponding ellipsoid
\(\mathcal E(1,1,c)\).

However, for the level set \( \mu = 0 \), we are free to choose a different foliation. In particular, we can select one consisting of simple closed curves that retract to a point; see Figure~\ref{fig:folX20}. This flexibility allows us to control the appearance and placement of singular fibers in the convex base diagram.

\begin{figure}[h!]
    \centering
    \begin{subfigure}[b]{0.7\textwidth}
        \centering
        \includegraphics[width=\textwidth]{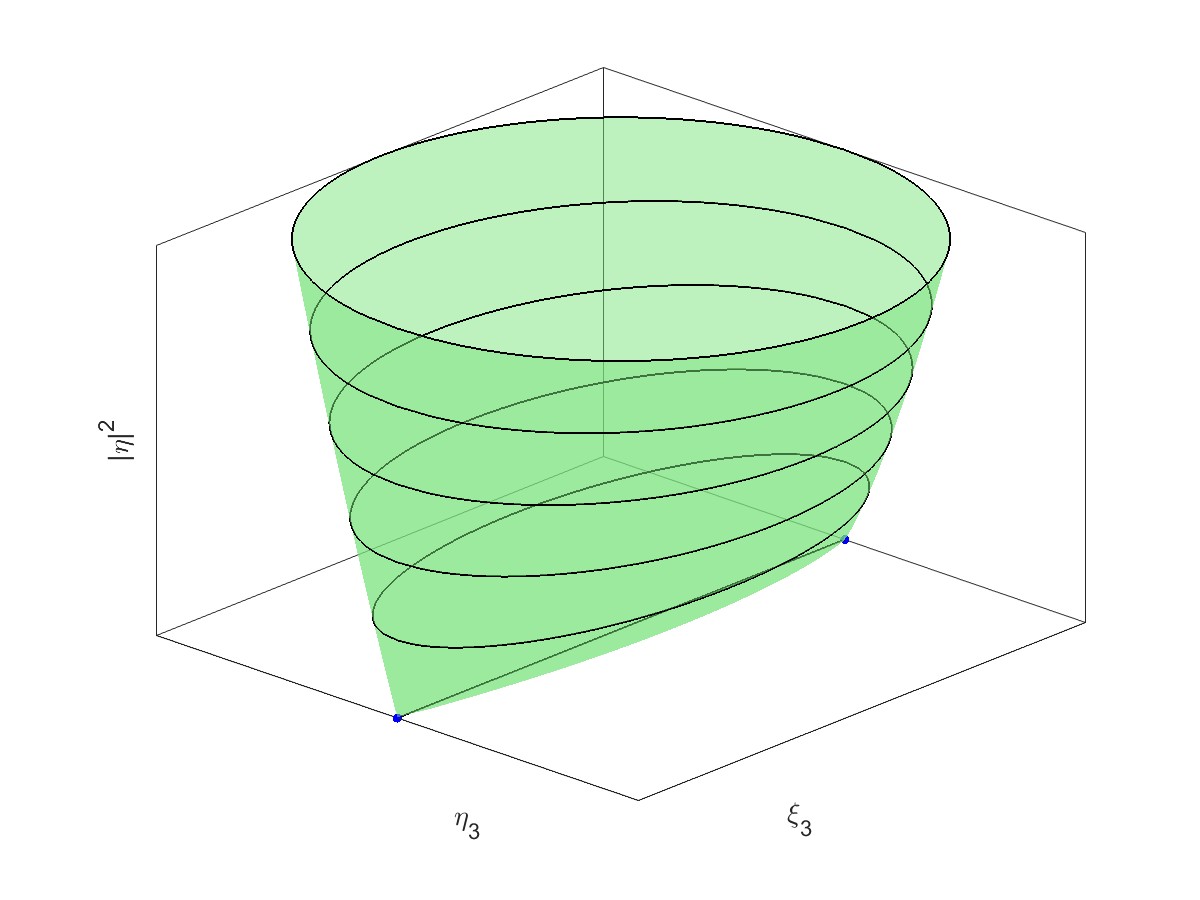}
        \caption{Foliation given by \( |\eta| \).}
    \end{subfigure}
    \hfill
    \begin{subfigure}[b]{0.8\textwidth}
        \centering
        \includegraphics[width=\textwidth]{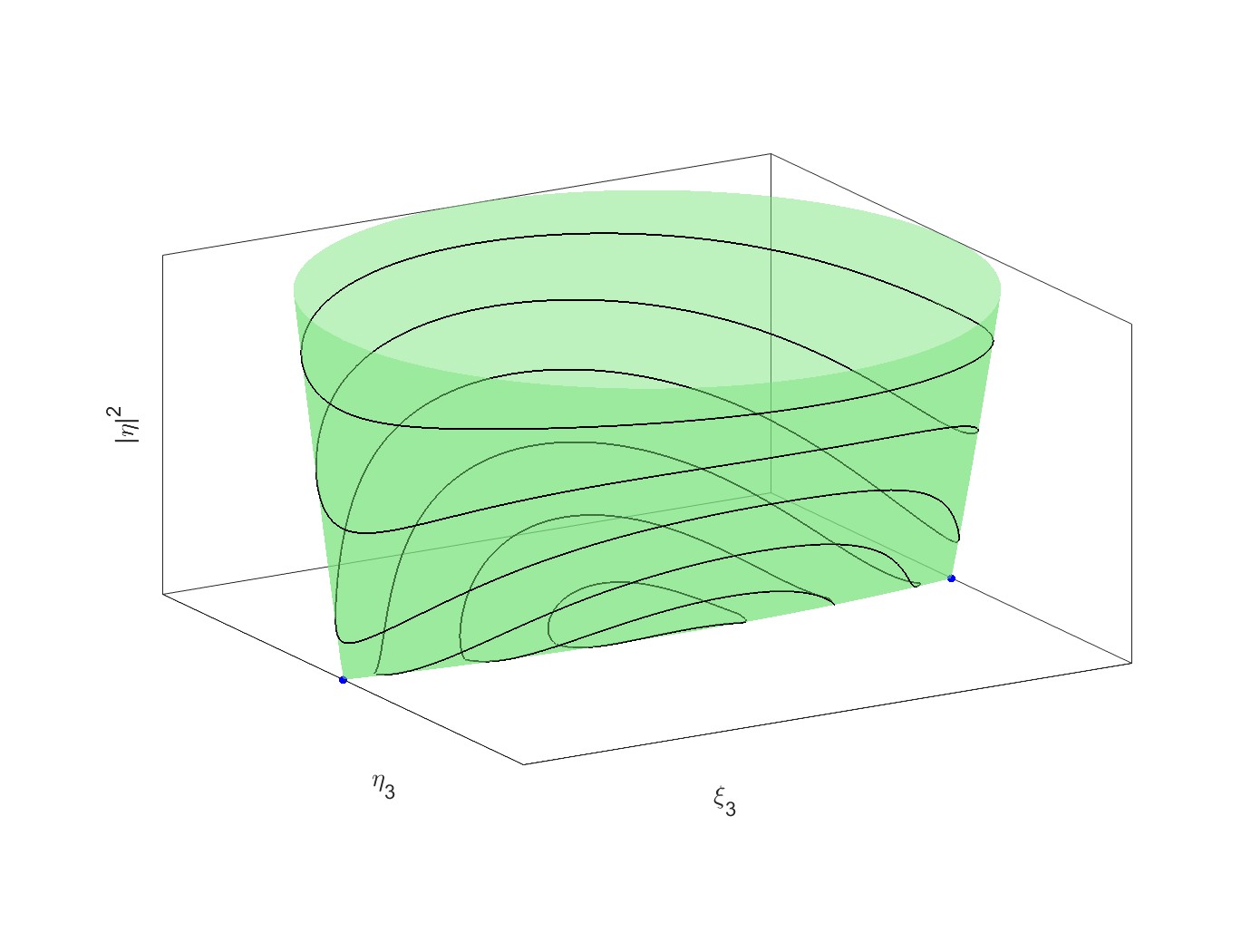}        
        \caption{Alternative foliation.}
    \end{subfigure}
    
    \caption{Different foliations for \( X_0^2 \).}
    \label{fig:folX20}
\end{figure}

As shown in the final part of the proof of Proposition~\ref{generalmeth2}, the singular fibers of the Lagrangian torus fibration correspond to base
points \(q\in X_\lambda^2\) for which the fiber
\(f_\lambda^{-1}(q)\) is not homeomorphic to \(S^1\). These singularities arise precisely at the fixed points of the circle action, where \( \xi_1 = \xi_2 = \eta_1 = \eta_2 = 0 \). From the defining equations of \( D^*S \), this implies \( \mu = u(\xi_3) = 0 \) and \( \eta_3 = 0 \), which only occur at the endpoints \( \xi_3 = a \) and \( \xi_3 = b \). These points are marked in blue in Figure~\ref{fig:folX20}.

To ensure that the convex base diagram contains only nodal singularities, we arrange the foliation on \( X^2_0 \) so that no leaf intersects both points \( (a, 0, 0) \) and \( (b, 0, 0) \) simultaneously. This setup yields exactly two nodal singularities; see Figure~\ref{fig:convex_base_11c_2}.

Next, we aim to determine the shape of the base diagram as in \cite[Section~2]{achigandrango2024singular}. For this purpose, we continue using the foliation given by the norm of the cotangent fibers. Note that changing the foliation does not affect the boundary of the base diagram; indeed, the boundary is determined by the limit of the curve as \( |\eta| \to 1 \).

Now, for \( \mu \neq 0 \) and fixed \( |\eta| \), we compute the symplectic area \( \mathcal{A}_S(\mu, |\eta|) \) of a disk inside \( D^*S \) whose boundary projects to the curve \( \gamma_{\mu, |\eta|} \) from part~\ref{claim: iii}. Observe that we can use the \( S^1 \)-action to rotate any point in \( f^{-1}_{\lambda}(\gamma_{\mu, |\eta|})  \) to one where \( \xi_1 = 0 \). There is exactly one loop in \( f^{-1}_{\lambda}(\gamma_{\mu, |\eta|}) \) satisfying equation~\eqref{eq: relation} with \( \xi_1 = 0 \). We will use this loop to compute the area of the disk.

We parameterize this loop using the coordinate \( \xi_3 \). To simplify notation, from now on we will
use $\mu=\xi_1\eta_2-\xi_2\eta_1$, omitting the factor of $2\pi$ present in the original definition of $\mu$. This is merely a rescaling of $\mu$, and we can reinstate the original definition at the end of the section if necessary. Assuming $\xi_3 \neq a, b$, equations \eqref{eq: relation}, \eqref{eq:2xi3}, \eqref{eq:2mu}, and \eqref{eq: T*S}, respectively, gives us that:
\begin{align*}
\eta_3 &= -\sqrt{\frac{u(\xi_3)^2 |\eta|^2 - \mu^2}{u(\xi_3)^2 (1 + u'(\xi_3)^2)}}, \\
\xi_2 &= u(\xi_3), \\
\eta_1 &= -\frac{\mu}{u(\xi_3)}, \\
\eta_2 &= -\sqrt{\frac{u(\xi_3)^2 |\eta|^2 - \mu^2}{u(\xi_3)^2 (1 + u'(\xi_3)^2)}} u'(\xi_3).
\end{align*}

Note that the equation $u(\xi_3)^2 |\eta|^2 - \mu^2 = 0$ has exactly two solutions except when, $|\eta|^2=\frac{\mu^2}{u(0)^2}$, in this case $\xi_3=0$ is the unique solution.  Let $\xi_{3_{\max}}^{\eta, \mu} > \xi_{3_{\min}}^{\eta, \mu}$ denote the two solutions to this equation. By Stokes' Theorem, the area of the disk $D_{\mu,|\eta|}$ with boundary projecting to the curve $\gamma_{\mu,|\eta|}$, is given by:

\begin{equation} \label{eq: splitting}
\begin{aligned}
    \mathcal{A}_S(\mu, |\eta|) &= \int_{D_{\mu,|\eta|}} \omega, \\
	 &= \int_{\partial D_{\mu,|\eta|}} -(\eta_1 d\xi_1 + \eta_2 d\xi_2 + \eta_3 d\xi_3), \\
	 &= \int_{\partial D_{\mu,|\eta|}} -(\eta_2 d\xi_2 + \eta_3 d\xi_3), \quad \text{since $\xi_1 = 0$}, \\
	 &= 2 \int_{\xi_{3_{\min}}^{\eta, \mu}}^{\xi_{3_{\max}}^{\eta, \mu}} - \eta_3\left( u'(\xi_3)^2+1 \right)d\xi_3, \quad \text{since $\eta_2 = u'(\xi_3)\eta_3$ and $\xi_2=u(\xi_3)$},\\
	 &= 2 \int_{\xi_{3_{\min}}^{\eta, \mu}}^{\xi_{3_{\max}}^{\eta, \mu}} \sqrt{\frac{u(\xi_3)^2 |\eta|^2 - \mu^2}{u(\xi_3)^2 (1 + u'(\xi_3)^2)}}( u'(\xi_3)^2+1)~d\xi_3.
\end{aligned}
\end{equation}

Then, the full moment map for the almost toric fibration is given by 
$$(\xi, \eta) \mapsto \left(\mu(\xi, \eta), \mathcal{A}_S(\mu(\xi, \eta), |\eta|)\right),$$ 
and the base $B_S$ of this fibration is bounded by the $x$-axis and the curve parametrized by $\mu \mapsto \left(\mu, \mathcal{A}_S(\mu, 1)\right)$, for $\mu \in [-u(0),u(0)]$.

We are now left to locate the singular values of this fibration and show that they correspond to nodal singularities as claimed. As we pointed out above, the singular values occur when $\mu = 0$, that is, over the positive $y$-semi-axis. As in Example \ref{ex: toy model}, there will be exactly two such singular values, also corresponding to nodal singularities, and careful computations can show exactly where they are located along the axis. However, the exact position along the $y$-axis is not too relevant here, given that these correspond to nodal singularities and the operation of nodal slide along the eigenline of the monodromy is allowed for such singularities, see \cite{symington71four}. These nodal slides correspond to a modification of the foliation used to describe the singular Lagrangian torus fibration.
\end{proof}

We now focus on the particular case of ellipsoids of revolutions of the form $\mathcal{E}(1,1,c)$. In this case we can give the precise description of the image of the base diagram of the almost toric fibration in Theorem \ref{thm: ATF_surface_revolution}. That is the content of Proposition \ref{prop: ellipsoids}, whose proof we now provide.

\begin{proof}[Proof of Proposition \ref{prop: ellipsoids}]
In the notation of Theorem \ref{thm: ATF_surface_revolution}, let us take:
\[ u(\xi_3) = \sqrt{1 - \frac{\xi_3^2}{c^2}}, \]
for $c > 0$. We use equation \eqref{eq: splitting}, obtained in the proof of Theorem \ref{thm: ATF_surface_revolution}, to get

\[\mathcal{A}_S(\mu, |\eta|) = 2 \int_{\xi_{3_{\min}}^{\eta, \mu}}^{\xi_{3_{\max}}^{\eta, \mu}} \sqrt{\frac{|\eta|^2(c^2-\xi_3^2) - c^2\mu^2}{c^2(c^2-\xi_3^2)+\xi_3^2}}~\frac{c^2(c^2-\xi_3^2)+\xi_3^2}{c(c^2-\xi_3^2)}~d\xi_3.\]

Notice that, in this case, we have that $\xi_{3_{\min}}^{\eta,\mu}=-c\sqrt{\frac{|\eta|^2-\mu^2}{|\eta|^2}}$ and $\xi_{3_{\max}}^{\eta,\mu}=c\sqrt{\frac{|\eta|^2-\mu^2}{|\eta|^2}}$. We now make the change of variable, using $\xi_3=c\sqrt{\frac{|\eta|^2-\mu^2}{|\eta|^2}}\cos \theta$. 
\begin{align*}
	\mathcal{A}_S(\mu, |\eta|) =& 2 |\eta| \int_{0}^{\pi} \frac{\sin^2\theta (|\eta|^2-\mu^2)(|\eta|^2+(c^2-1)(\mu^2+(|\eta|^2-\mu^2)\sin^2\theta))}{(\mu^2+(|\eta|^2-\mu^2)\sin^2\theta)\sqrt{|\eta|^2+(c^2-1)(\mu^2+(|\eta|^2-\mu^2)\sin^2\theta)}} ~d\theta\\
	=&4  |\eta| \left( \int_{0}^{\pi/2} \sqrt{|\eta|^2+(c^2-1)(\mu^2+(|\eta|^2-\mu^2)\sin^2\theta)} ~d\theta \right. \\
	& -(c^2-1) \mu^2 \int_{0}^{\pi/2} \frac{1}{ \sqrt{|\eta|^2+(c^2-1)(\mu^2+(|\eta|^2-\mu^2)\sin^2\theta)}}~d\theta\\
	& -\mu^2|\eta^2|\left. \int_{0}^{\pi/2} \frac{1}{(\mu^2+(|\eta|^2-\mu^2)\sin^2\theta)\sqrt{|\eta|^2+(c^2-1)(\mu^2+(|\eta|^2-\mu^2)\sin^2\theta)}} ~d\theta\right).
	\end{align*}

For $\mu\neq 0$, make $D(\mu,\eta):=\frac{|\eta|^2-\mu^2}{\mu^2}$ and  $B_c(\mu,\eta):=|\eta|^2+(c^2-1)\mu^2$, for every $c>0$. Then,

\begin{equation}\label{eq:bound_cbd_11c}
\begin{aligned}
\mathcal{A}_S(\mu,|\eta|)
&=
4|\eta|\Bigg(
\sqrt{B_c(\mu,\eta)}
E\left(
1-\frac{|\eta|^2c^2}{B_c(\mu,\eta)}
\right)
\\
&\qquad
-\frac{(c^2-1)\mu^2}{\sqrt{B_c(\mu,\eta)}}
F\left(
1-\frac{|\eta|^2c^2}{B_c(\mu,\eta)}
\right)
\\
&\qquad
-\frac{|\eta|^2}{\sqrt{B_c(\mu,\eta)}}
\Pi\left(
-D(\mu,\eta),
1-\frac{|\eta|^2c^2}{B_c(\mu,\eta)}
\right)
\Bigg).
\end{aligned}
\end{equation}

In the last equation, we are using the definition of the complete elliptical integrals of first, second and third kinds. We now recall their definitions. For $k,n<1$, let
\begin{equation*}
\begin{aligned}
F(k)&=F\left(\frac{\pi}{2}\Bigm\vert k\right)=\int_0^{\pi/2}\frac{d\theta}{\sqrt{1-k\sin^2\theta}},\\
E(k)&=E\left(\frac{\pi}{2}\Bigm\vert k\right)=\int_0^{\pi/2}\sqrt{1-k\sin^2\theta}\,d\theta,\\
\Pi(n,k)&=\Pi\left(n;\frac{\pi}{2}\Bigm\vert k\right)=\int_0^{\pi/2}\frac{d\theta}{(1-n\sin^2\theta)\sqrt{1-k\sin^2\theta}}.
\end{aligned}
\end{equation*}

\begin{figure}[h!]
    \centering
    \begin{subfigure}[b]{0.32\textwidth}
        \centering
        \includegraphics[width=\textwidth]{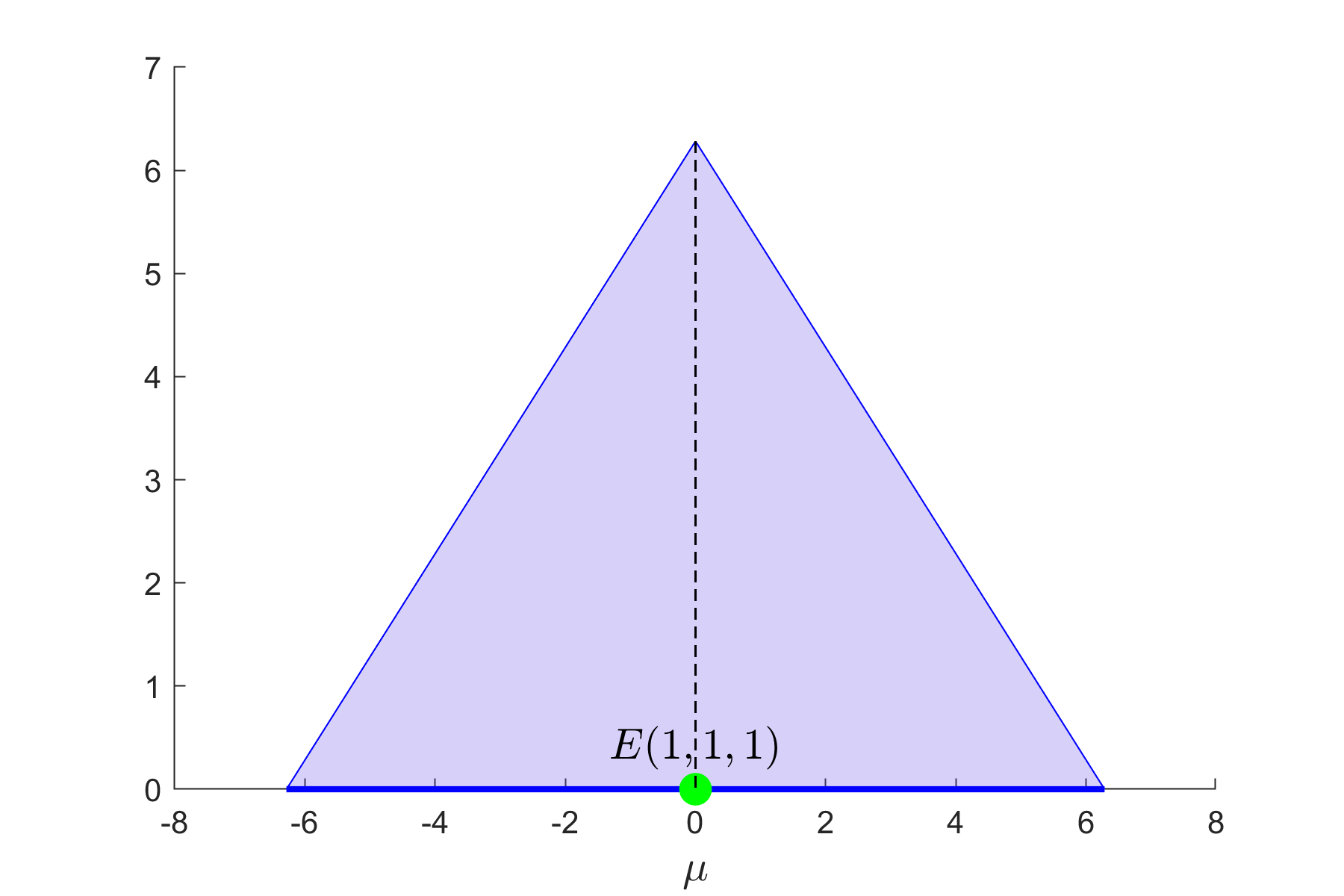}
        \caption{$c=1$}
        \label{fig:convex_111}
    \end{subfigure}
    \hfill
    \begin{subfigure}[b]{0.32\textwidth}
        \centering
        \includegraphics[width=\textwidth]{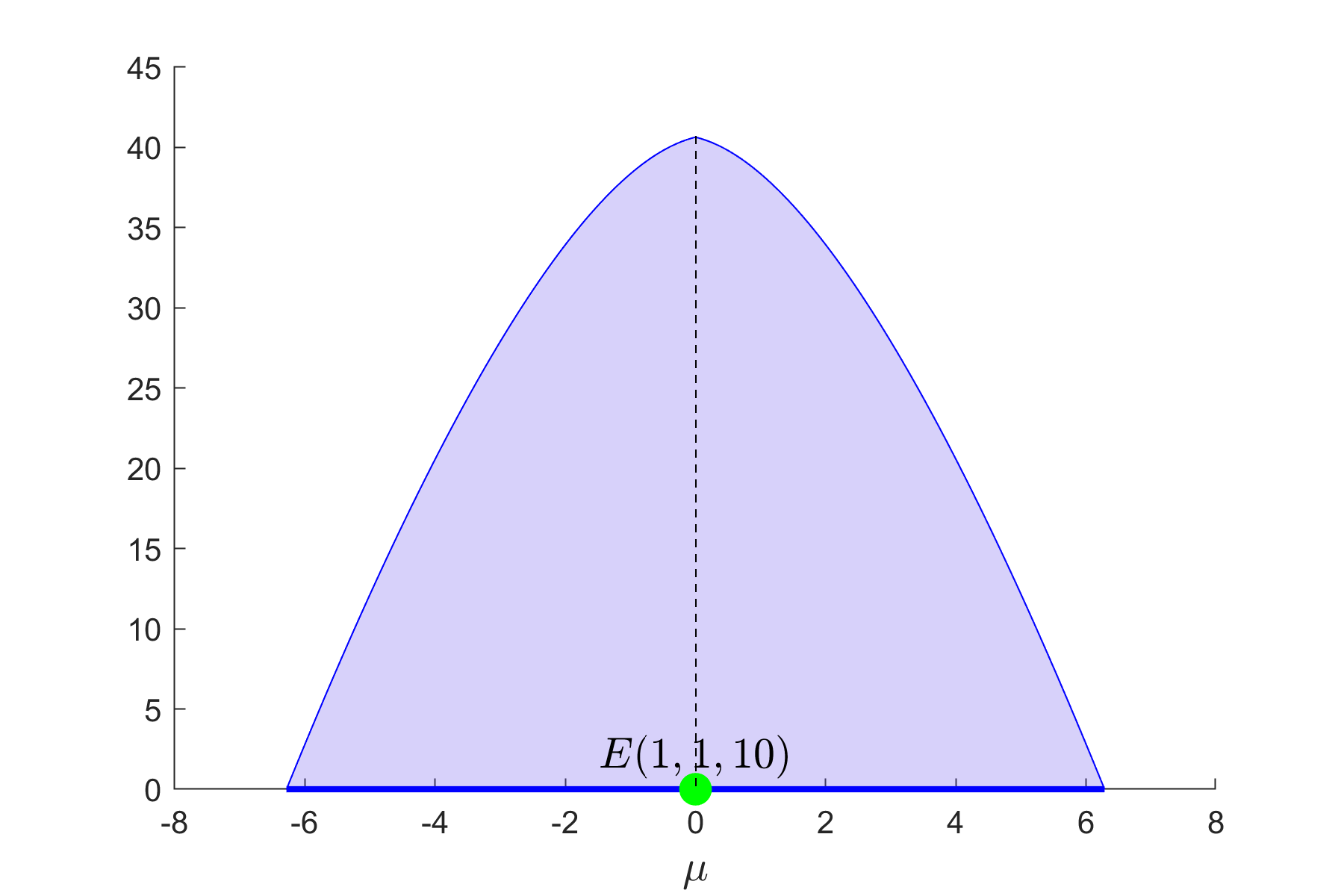}
        \caption{$c=10$}
        \label{fig:convex_1110}
    \end{subfigure}
    \hfill
    \begin{subfigure}[b]{0.32\textwidth}
        \centering
        \includegraphics[width=\textwidth]{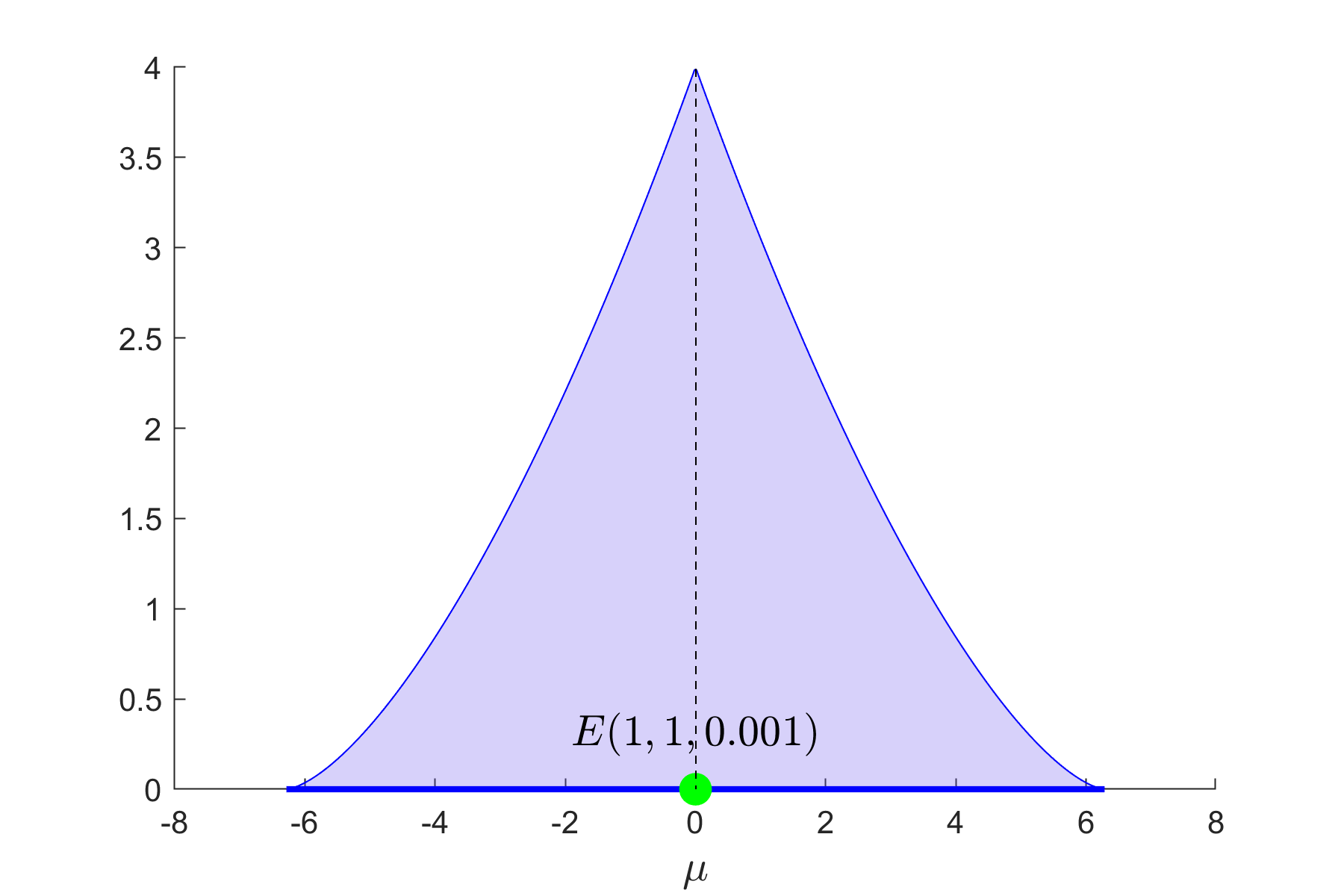}
        \caption{$c=0.001$}
        \label{fig:convex_1105}
    \end{subfigure}
    \caption{Base diagram of $D^*\mathcal{E}(1,1,c)$.}
    \label{fig: convex_base_11c}
\end{figure}

This concludes our computations for the symplectic area of the disk and the characterization of the base diagram. Note that using the foliation given by the norm of the cotangent fibers, we obtain a base diagram where the zero section of the cotangent bundle maps to a point, see Figure \ref{fig: convex_base_11c}. However, by modifying the foliation of $X^2_0$, we can obtain a convex base diagram with two nodal singularities as explained in the proof of Theorem \ref{thm: ATF_surface_revolution}, see Figure \ref{fig:convex_base_11c_2}.
\end{proof}

\subsection{Disk cotangent bundle of 3-dimensional  ellipsoids $\mathcal{E}(1,1,c,c)$}
\begin{proof}[Proof of Theorem \ref{thm: highellips}]

Let's consider $ D^*\mathcal{E}(1,1,c,c) \subset \mathbb{C}^4$ given by:

\begin{equation}\label{eq11cc}
D^*\mathcal{E}(1,1,c,c):= \left\{
  \xi +i \eta \in \mathbb{C}^4 \;\middle|\;
  \begin{aligned}
  & \xi_1^2+\xi_2^2+ \frac{\xi_3^2}{c^2}+\frac{\xi_4^2}{c^2}=1 \\
  & \xi_1\eta_1+\xi_2\eta_2+ \frac{\xi_3\eta_3}{c^2}+\frac{\xi_4\eta_4}{c^2}=0\\
  & |\eta|^2<1
  \end{aligned}
\right\},
\end{equation}
with the symplectic form: $\omega=\sum d\xi_j \wedge d\eta_j$. Define the $T^2$-action in $D^*\mathcal{E}(1,1,c,c)$ given by:
\[
\scalebox{0.95}{$
\begin{pmatrix}
	\xi_1\\
	\xi_2\\
	\xi_3\\
	\xi_4\\
	\eta_1\\
	\eta_2\\
	\eta_3\\
	\eta_4
\end{pmatrix} \longmapsto
\begin{pmatrix}
	\cos 2\pi \theta_1 & -\sin 2\pi\theta_1 & 0 & 0 & 0 & 0& 0 & 0 \\
	\sin 2\pi\theta_1 & \cos 2\pi\theta_1 & 0 & 0 & 0 & 0& 0 & 0\\
	0&0&\cos 2\pi\theta_2 & -\sin 2\pi\theta_2&0&0&0& 0\\
	0&0&\sin 2\pi\theta_2 & \cos 2\pi\theta_2&0&0&0& 0\\
	0&0&0&0& \cos2\pi\theta_1 & -\sin 2\pi\theta_1 & 0&0 \\
	0&0&0&0& \sin 2\pi\theta_1 & \cos 2\pi\theta_1 & 0&0 \\
	0&0&0&0&0&0&\cos2\pi\theta_2 & -\sin 2\pi\theta_2\\
	0&0&0&0&0&0&\sin2\pi\theta_2 & \cos 2\pi\theta_2\\
\end{pmatrix}
\begin{pmatrix}
	\xi_1\\
	\xi_2\\
	\xi_3\\
	\xi_4\\
	\eta_1\\
	\eta_2\\
	\eta_3\\
	\eta_4
\end{pmatrix}
$}
\]

This action has moment map
\[
\mu=(\mu_1,\mu_2)
:=
2\pi(\xi_1\eta_2-\xi_2\eta_1,\xi_3\eta_4-\xi_4\eta_3).
\]
Using this definition, the Cauchy--Schwarz inequality and \eqref{eq11cc}
give
\[
\left|
\frac{\mu_1}{2\pi}\pm\frac{\mu_2}{2\pi c}
\right|
=
\left|
\xi_1\eta_2-\xi_2\eta_1
\pm
\frac{\xi_3\eta_4-\xi_4\eta_3}{c}
\right|
\leq
|\eta|
\sqrt{
\xi_1^2+\xi_2^2+\frac{\xi_3^2}{c^2}+\frac{\xi_4^2}{c^2}
}
=
|\eta|.
\]
Therefore
\begin{equation}\label{maxeta_he}
|\eta|
\geq
\frac{1}{2\pi}
\max
\left\{
\left|\mu_1+\frac{\mu_2}{c}\right|,
\left|\mu_1-\frac{\mu_2}{c}\right|
\right\}.
\end{equation}

This inequality will be useful later. We consider the map defined by
\begin{align*}
	f: D^*\mathcal{E}(1,1,c,c) &\longrightarrow \mathbb{R}^3,\\
	(\xi,\eta) &\longmapsto (t_1,t_2,|\eta|^2)\\
	&:=
	\left(
	(\xi_1^2+\xi_2^2)-(\eta_1^2+\eta_2^2),
	2(\xi_1\eta_1+\xi_2\eta_2),
	|\eta|^2
	\right).
\end{align*}

We define \(f_\lambda:=f|_{\mu^{-1}(\lambda)}\). To use Proposition
\ref{generalmeth2}, we need to prove that
\(f(D^*\mathcal{E}(1,1,c,c)\cap\mu^{-1}(\lambda))\) is simply connected for
every \(\lambda\), and that the induced map \(\widetilde f_\lambda\) is a
bijection. In the following claim, we prove that
\(f(D^*\mathcal{E}(1,1,c,c)\cap\mu^{-1}(\lambda))\) is indeed simply
connected and provide a foliation of this set given by the norm of the
cotangent fibers. By abuse of notation, we refer to fixed values of the
moment map and a fixed value of the norm of the cotangent fibers as
\(\mu_1,\mu_2\), and \(|\eta|\), respectively.

\begin{clm} \label{clm: highellips}
\begin{enumerate}[label=(\roman*)]
    \item \label{claim_he: i} The following relation holds:
    \begin{equation} \label{eq: highellips}
    \left(
    \frac{\mu_1^2}{2\pi^2}
    -
    \frac{\mu_2^2}{2\pi^2c^2}
    +
    2|\eta|^2
    +
    t_1(1-|\eta|^2)
    +
    \frac{(1-c^2)t_2^2}{2}
    \right)^2
    =
    (1+|\eta|^2)^2
    \left(
    t_1^2+t_2^2+\frac{\mu_1^2}{\pi^2}
    \right).
    \end{equation}
    
    \item \label{claim_he: ii} For fixed \(\mu_1,\mu_2\) and
    \(|\eta|\neq 0\), the locus in the plane \((t_1,t_2)\) defined by
    equation \eqref{eq: highellips} is a closed curve or a point. For
    \(|\eta|=0\), we have \(\mu_1=\mu_2=0\), and the locus is
    \([0,1]\times\{0\}\), i.e. the interval \([0,1]\) along the \(t_1\)-axis.
    
    \item \label{claim_he: iii} The image
    \(f(D^*\mathcal{E}(1,1,c,c)\cap\mu^{-1}(\lambda))\) is simply connected
    and can be foliated using the norm of the cotangent fibers.
    
    \item \label{claim_he: iv} The induced map \(\widetilde f_\lambda\) is a
    bijection onto its image.
\end{enumerate}
\end{clm}

\begin{proof}[Proof of Claim \ref{clm: highellips}]

\noindent\textbf{(i)} To prove the relation \eqref{eq: highellips}, consider
the following equations derived from the system:
\begin{align}
    \frac{\mu_1^2}{4\pi^2}+\frac{t_2^2}{4}
    &=
    (\xi_1^2+\xi_2^2)(\eta_1^2+\eta_2^2),
    \label{eq_he:eta1eta2p1}
    \\
    \xi_1^2+\xi_2^2
    &=
    t_1+(\eta_1^2+\eta_2^2),
    \label{eq_he:eta1eta2p2}
    \\
    \frac{\mu_2^2}{4\pi^2c^4}+\frac{t_2^2}{4}
    &=
    \frac{(\xi_3^2+\xi_4^2)(\eta_3^2+\eta_4^2)}{c^4},
    \label{eq_he:eta3eta4p1}
    \\
    \frac{\xi_3^2+\xi_4^2}{c^2}
    &=
    (\eta_3^2+\eta_4^2)+1-|\eta|^2-t_1.
    \label{eq_he:eta3eta4p2}
\end{align}

Solving the system for \(\eta_1^2+\eta_2^2\) and
\(\eta_3^2+\eta_4^2\) yields
\begin{align}
    \eta_1^2+\eta_2^2
    &=
    \frac{-t_1+\sqrt{t_1^2+t_2^2+\frac{\mu_1^2}{\pi^2}}}{2},
    \label{eq_he:normeta1}
    \\
    \eta_3^2+\eta_4^2
    &=
    \frac{
    t_1+|\eta|^2-1+
    \sqrt{
    (t_1+|\eta|^2-1)^2
    +
    \frac{\mu_2^2}{\pi^2c^2}
    +
    c^2t_2^2
    }
    }{2}.
    \label{eq_he:normeta2}
\end{align}

Adding \eqref{eq_he:normeta1} and \eqref{eq_he:normeta2}, and performing
algebraic simplifications, gives equation \eqref{eq: highellips}.

\vspace{1em}
\noindent\textbf{(ii)} For \(|\eta|=0\), it follows that
\(\mu_1=\mu_2=t_2=0\) and \(t_1\in[0,1]\), hence the locus is
\([0,1]\times\{0\}\).

For \(|\eta|\neq 0\), equation \eqref{eq: highellips} can be rewritten as
\begin{equation} \label{eq:closed_curve}
\begin{aligned}
&\frac{1}{1+|\eta|^2}
\left(
2|\eta|^2t_1
-
(1-|\eta|^2)
\left(
\frac{\mu_1^2}{4\pi^2}
-
\frac{\mu_2^2}{4\pi^2c^2}
+
|\eta|^2
+
\frac{(1-c^2)t_2^2}{4}
\right)
\right)^2
\\
&\qquad =
\left(
\frac{\mu_1^2}{4\pi^2}
-
\frac{\mu_2^2}{4\pi^2c^2}
+
|\eta|^2
+
\frac{(1-c^2)t_2^2}{4}
\right)^2
-
|\eta|^2
\left(
t_2^2+\frac{\mu_1^2}{\pi^2}
\right).
\end{aligned}
\end{equation}

When \(c=1\), this is the equation of an ellipse. For \(c\neq 1\), the
right-hand side of \eqref{eq:closed_curve} is a quartic polynomial in \(t_2\)
that tends to \(+\infty\) as \(t_2\to\pm\infty\). Due to its structure, it has
at most four real roots, and the region where it is non-negative consists of a
compact set, containing \(t_2=0\), and two unbounded components. In our
analysis, we focus on the compact set around \(t_2=0\), where the polynomial
is non-negative and the corresponding solution set is well-defined:
\begin{equation} \label{eq:intervalt2}
|t_2|
\leq
2\sqrt{
\frac{
\left(
\frac{\mu_2^2}{4\pi^2c^2}
-
\frac{\mu_1^2}{4\pi^2}
\right)(1-c^2)
+
|\eta|
\left(
|\eta|(1+c^2)
-
2\sqrt{
\left(
\frac{\mu_2^2}{4\pi^2c^2}
-
\frac{\mu_1^2c^2}{4\pi^2}
\right)(1-c^2)
+
c^2|\eta|^2
}
\right)
}{
(1-c^2)^2
}
}.
\end{equation}
This interval degenerates to zero whenever
\[
|\eta|
=
\left|
\frac{\mu_1}{2\pi}
\pm
\frac{\mu_2}{2\pi c}
\right|.
\]
By \eqref{maxeta_he},
\[
|\eta|
\in
\left[
\frac{1}{2\pi}
\max
\left\{
\left|\mu_1+\frac{\mu_2}{c}\right|,
\left|\mu_1-\frac{\mu_2}{c}\right|
\right\},
1
\right),
\]
so the interval defined by the inequality in \eqref{eq:intervalt2} collapses
to a point whenever
\[
|\eta|
=
\frac{1}{2\pi}
\max
\left\{
\left|\mu_1+\frac{\mu_2}{c}\right|,
\left|\mu_1-\frac{\mu_2}{c}\right|
\right\}.
\]

For each \(t_2\) in the interior of the interval given by
\eqref{eq:intervalt2}, where the right-hand side of \eqref{eq:closed_curve}
is strictly positive, there are exactly two corresponding values of \(t_1\)
satisfying the equation. At the endpoints of the interval, where the
polynomial vanishes, there is exactly one solution for \(t_1\). Since the
right-hand side depends continuously on \(t_2\), and the solutions change
continuously, the full set of solutions to \eqref{eq:closed_curve} for
\(t_2\) in the interval \eqref{eq:intervalt2} forms a simple closed curve in
the \((t_1,t_2)\)-plane. We denote these simple closed curves by
\(\gamma_{\mu,|\eta|}\).

\vspace{1em}
\noindent\textbf{(iii)} This follows directly from (ii). The sets
\(\gamma_{\mu,|\eta|}\) defined by equation \eqref{eq:closed_curve} form a
family of closed curves. As \(|\eta|\) becomes smaller, these closed curves
retract to a line if \(\mu_1=\mu_2=0\), or to a point otherwise. This gives a
foliation of the image by level sets of the norm in the cotangent fiber.

\vspace{1em}
\noindent\textbf{(iv)} To prove the injectivity of \(\widetilde f_\lambda\),
we show that the preimage of each point consists of exactly one
\(T^2\)-orbit.

From equations \eqref{eq_he:normeta1}, \eqref{eq_he:normeta2},
\eqref{eq_he:eta1eta2p2}, and \eqref{eq_he:eta3eta4p2}, we see that, for
fixed \(\mu_1,\mu_2\), the quantities
\[
\xi_1^2+\xi_2^2,\qquad
\eta_1^2+\eta_2^2,\qquad
\xi_3^2+\xi_4^2,\qquad
\eta_3^2+\eta_4^2
\]
are uniquely determined by \(t_1,t_2\), and \(|\eta|\).

Consider first the \(S^1\)-action associated to \(\mu_1\). If
\(\xi_1^2+\xi_2^2=0\), then the base component vanishes, and the fiber norm
\(\eta_1^2+\eta_2^2\) is fixed; thus, the orbit is determined by rotations in
the plane \((\eta_1,\eta_2)\), yielding a unique \(S^1\)-orbit. If instead
\(\xi_1^2+\xi_2^2\neq 0\), we may use the \(S^1\)-action to set
\(\xi_1=0\). In this case, the conditions
\[
\eta_1=-\frac{\mu_1}{2\pi\xi_2},
\qquad
\eta_2=\frac{t_2}{2\xi_2}
\]
determine the fiber coordinates uniquely, so again, the orbit is uniquely
specified.

A completely analogous argument applies to the \(S^1\)-action associated to
\(\mu_2\). Indeed, if \(\xi_3^2+\xi_4^2\neq 0\), we may use this circle action
to set \(\xi_3=0\). Then
\[
\eta_3=-\frac{\mu_2}{2\pi\xi_4},
\qquad
\eta_4=-\frac{c^2t_2}{2\xi_4},
\]
where the second identity follows from the cotangent condition in
\eqref{eq11cc}. If \(\xi_3^2+\xi_4^2=0\), the orbit is again determined by
the corresponding fiber norm. Hence each point in the image has a unique
\(T^2\)-orbit in its preimage, establishing that \(\widetilde f_\lambda\) is
injective.
\end{proof}

To construct an almost toric fibration on \(D^*\mathcal{E}(1,1,c,c)\), we
start by examining the foliation induced by the norm of the cotangent fibers.
This foliation collapses the zero section to a point in the base diagram,
similarly to the case of \(D^*S\). However, for the level set
\(\mu_1=\mu_2=0\), we have the freedom to choose a different foliation. In
particular, we may select a foliation consisting of simple closed curves that
retract to a point, which gives us flexibility in shaping the base diagram
and controlling the singularities.

According to the final part of the proof of Proposition~\ref{generalmeth2},
the singular fibers of the Lagrangian torus fibration correspond to those base
points \(q\) for which \(f^{-1}(q)\cap\mu^{-1}(\lambda)\) is not
homeomorphic to \(T^2\). These singularities are associated with fixed points
of the circle actions generating \(\mu_1\) and \(\mu_2\).

The fixed points of the circle action associated with \(\mu_1\) occur where
\[
\xi_1=\xi_2=\eta_1=\eta_2=0.
\]
From the defining equations~\eqref{eq11cc}, this implies
\[
\mu_1=0,\qquad t_1=t_2=0,
\]
while \(\mu_2\) varies in the interval \((-2\pi c,2\pi c)\). More precisely,
for fixed \(\mu_2\), this fixed-point locus maps in \(X_\lambda^2\) to
\[
(t_1,t_2,|\eta|^2)
=
\left(
0,
0,
\frac{\mu_2^2}{4\pi^2c^2}
\right).
\]

On the other hand, the fixed points of the circle action associated with
\(\mu_2\) satisfy
\[
\xi_3=\xi_4=\eta_3=\eta_4=0.
\]
Thus
\[
\mu_2=0,
\]
while \(\mu_1\) varies in the interval \((-2\pi,2\pi)\). For fixed
\(\mu_1\), this fixed-point locus maps in \(X_\lambda^2\) to
\[
(t_1,t_2,|\eta|^2)
=
\left(
1-\frac{\mu_1^2}{4\pi^2},
0,
\frac{\mu_1^2}{4\pi^2}
\right).
\]

In particular, over the zero moment level \(\mu_1=\mu_2=0\), these two fixed
loci give the two points
\[
(0,0,0)
\qquad\text{and}\qquad
(1,0,0)
\]
in \(X_0^2\).

To ensure the base diagram contains only nodal singularities, we choose a
foliation on
\[
f(D^*\mathcal{E}(1,1,c,c)\cap\mu^{-1}(0))
\]
consisting of simple closed curves that retract to a point, and such that no
curve intersects both \((0,0,0)\) and \((1,0,0)\). Note that in this case, the
\(T^2\)-action does not collapse to a point, since \(|\xi|=0\) is never
attained on this space.

We now proceed to determine the shape of the base diagram, following the
construction in~\cite[Section 2]{achigandrango2024singular}. For this
purpose, we return to using the foliation induced by the norm of the cotangent
fibers. It is important to observe that the boundary of the base diagram
remains unchanged under different choices of foliation; indeed, it is
determined by the image of the hypersurface \(|\eta|=1\).

Now we compute, for \(|\eta|\) fixed, the symplectic area
\(\mathcal A_c(\mu,|\eta|)\) of a disk inside
\(D^*\mathcal{E}(1,1,c,c)\) with boundary on the Lagrangian
\(L_{\mu,|\eta|}\) that projects to the curve \(\gamma_{\mu,|\eta|}\) in
\ref{claim_he: iii}. Notice that we can act with the \(T^2\)-action on
\(f^{-1}(q)\cap\mu^{-1}(\lambda)\) and assume \(\xi_1=\xi_3=0\). There is
only one loop inside \(f^{-1}(\gamma_{\mu,|\eta|})\cap\mu^{-1}(\lambda)\)
satisfying \eqref{eq11cc} with \(\xi_1=\xi_3=0\). We will use this loop to
compute the area of the disk.

\begin{align*}
	\mathcal A_c(\mu,|\eta|)
	=
	\int_D\omega
	&=
	\int_{\partial D}
	(\xi_1d\eta_1+\xi_2d\eta_2+\xi_3d\eta_3+\xi_4d\eta_4),
	\\
	&=
	\int_{\partial D}(\xi_2d\eta_2+\xi_4d\eta_4),
	\quad \text{since \(\xi_1=\xi_3=0\)}.
\end{align*}
By using the defining equations for the case \(\xi_1=\xi_3=0\), we get
\[
\eta_4
=
-\sqrt{
\frac{
\left(
|\eta|^2c^2\xi_4^2
-
\frac{\mu_2^2c^2}{4\pi^2}
\right)(c^2-\xi_4^2)
-
\frac{c^4\mu_1^2\xi_4^2}{4\pi^2}
}{
\xi_4^2\left(\xi_4^2+c^2(c^2-\xi_4^2)\right)
}
},
\]
\[
\xi_2=\sqrt{1-\frac{\xi_4^2}{c^2}},
\qquad
\eta_2=-\frac{\xi_4\eta_4}{c^2\xi_2}.
\]
From this, we obtain an integral over an interval inside the possible values
of \(\xi_4\). The interval that we use is
\((\xi_4^{\min},\xi_4^{\max})\), where
\[
\xi_4^{\min}
=
\frac{1}{|\eta|}
\sqrt{
\frac{
\frac{\mu_2^2}{4\pi^2}
+
c^2
\left(
|\eta|^2-\frac{\mu_1^2}{4\pi^2}
\right)
-
\sqrt{
-4\frac{\mu_2^2}{4\pi^2}c^2|\eta|^2
+
\left(
\frac{\mu_2^2}{4\pi^2}
+
c^2
\left(
|\eta|^2-\frac{\mu_1^2}{4\pi^2}
\right)
\right)^2
}
}{2}
},
\]
and
\[
\xi_4^{\max}
=
\frac{1}{|\eta|}
\sqrt{
\frac{
\frac{\mu_2^2}{4\pi^2}
+
c^2
\left(
|\eta|^2-\frac{\mu_1^2}{4\pi^2}
\right)
+
\sqrt{
-4\frac{\mu_2^2}{4\pi^2}c^2|\eta|^2
+
\left(
\frac{\mu_2^2}{4\pi^2}
+
c^2
\left(
|\eta|^2-\frac{\mu_1^2}{4\pi^2}
\right)
\right)^2
}
}{2}
}.
\]
Let
\[
A:=A(\mu_1,\mu_2,|\eta|,c)
:=
-4\frac{\mu_2^2}{4\pi^2}|\eta|^2c^2
+
\left(
\frac{\mu_2^2}{4\pi^2}
+
c^2
\left(
|\eta|^2-\frac{\mu_1^2}{4\pi^2}
\right)
\right)^2,
\]
and
\[
B:=B(\mu_1,\mu_2,|\eta|,c)
:=
\frac{\mu_2^2}{4\pi^2}
-
\frac{\mu_1^2c^2}{4\pi^2}
+
\sqrt A
+
c^2|\eta|^2.
\]

After direct computations, we get
\begin{equation}\label{eq: boundary_11cc}
\begin{aligned}
    \mathcal A_c(\mu,|\eta|)
    &=
    2
    \int_{\xi_4^{\min}}^{\xi_4^{\max}}
    \frac{
    \sqrt{
    \left(
    \left(|\eta|^2\xi_4^2-\frac{\mu_2^2}{4\pi^2}\right)(c^2-\xi_4^2)
    -
    \frac{c^2\mu_1^2\xi_4^2}{4\pi^2}
    \right)
    \left(\xi_4^2+c^2(c^2-\xi_4^2)\right)
    }
    }{
    c\xi_4(c^2-\xi_4^2)
    }
    \,d\xi_4
    \\
    &=
    \frac{\sqrt{2}}{c}
    \sqrt{
    \frac{1}{
    (1-c^2)B+2c^4|\eta|^2
    }
    }
    \Bigg(
    \left((1-c^2)B+2c^4|\eta|^2\right)
    E\left(
    \frac{
    2(1-c^2)\sqrt A
    }{
    (1-c^2)B+2c^4|\eta|^2
    }
    \right)
    \\
    &\qquad
    -
    2(1-c^2)
    \left(
    \frac{\mu_2^2}{4\pi^2}
    -
    \frac{\mu_1^2c^2}{4\pi^2}
    \right)
    F\left(
    \frac{
    2(1-c^2)\sqrt A
    }{
    (1-c^2)B+2c^4|\eta|^2
    }
    \right)
    \\
    &\qquad
    -
    \frac{
    4c^4|\eta|^2\frac{\mu_2^2}{4\pi^2}
    }{B}
    \Pi\left(
    \frac{2\sqrt A}{B},
    \frac{
    2(1-c^2)\sqrt A
    }{
    (1-c^2)B+2c^4|\eta|^2
    }
    \right)
    \\
    &\qquad
    +
    \frac{
    4c^4|\eta|^2\frac{\mu_1^2}{4\pi^2}
    }{
    B-2c^2|\eta|^2
    }
    \Pi\left(
    \frac{2\sqrt A}{B-2c^2|\eta|^2},
    \frac{
    2(1-c^2)\sqrt A
    }{
    (1-c^2)B+2c^4|\eta|^2
    }
    \right)
    \Bigg).
\end{aligned}
\end{equation}
\end{proof}

The image of the base diagram
\[
(\mu_1,\mu_2,\mathcal A_c(\mu,|\eta|))
\]
for \(|\eta|\leq 1\) was generated using \textsc{Matlab} and is displayed in
Figure~\ref{fig: convex_base_11cc}.

\subsection{The Lagrangian bidisk}
\begin{proof}[Proof of Theorem \ref{thm: lagbid}]

Recall that we define the Lagrangian bidisk as
\[
P_L = \{(p_1, p_2, q_1, q_2) \in \mathbb{R}^4 \mid p_1^2 + p_2^2 \leq 1,\ q_1^2 + q_2^2 \leq 1\},
\]
equipped with the standard symplectic form \( \omega = \sum dp_j \wedge dq_j \).

Define the circle action in $P_L$, given by:
\[ \begin{pmatrix}
	p_1\\
	p_2\\
	q_1\\
	q_2
\end{pmatrix} \longmapsto \begin{pmatrix}
	\cos2\pi\theta & -\sin 2\pi\theta & 0 & 0  \\
	\sin 2\pi\theta & \cos 2\pi\theta & 0 & 0 \\
	0&0&\cos2\pi\theta & -\sin 2\pi\theta\\
	0&0& \sin 2\pi\theta & \cos 2\pi\theta\\
\end{pmatrix} \begin{pmatrix}
	p_1\\
	p_2\\
	q_1\\
	q_2
\end{pmatrix}.\] 
This action has moment map $\mu = 2\pi( p_1 q _2 - p_2 q_1)$. Consider the fibration:
\begin{align*}
	f:P_L&\longrightarrow \mathbb{R}^2,\\
	(p,q)& \longmapsto (t_1,t_2),
\end{align*} where
\begin{align*}
t_1&=p_1q_1+p_2q_2,\\
	t_2&=p_1^2+p_2^2-q_1^2-q_2^2=|p|^2-|q|^2.
\end{align*}
Note that the fibers of this map are preserved by the circle action described above.

To use Proposition \ref{generalmeth2}, we need to prove that $f(P_L \cap \mu^{-1}(\lambda))$ is simply connected for every $\lambda\in [-2\pi,2\pi]$, and that $\widetilde{f}_\lambda$ is injective. With the following claim, we prove the necessary conditions. By abuse of notation, we refer to fixed values of the moment map and the norms of $p$ and $q$ as $\mu$, $|p|$, and $|q|$, respectively.

\begin{clm} \label{clm: lagbi}
\begin{enumerate}[label=(\roman*)]
    \item \label{claim_lagbi: i} The following relation holds:
    \begin{equation} \label{eq: lagbi}
        \mu^2 + 4\pi^2 t_1^2 = 4\pi^2 |p|^2 |q|^2.
    \end{equation}
    
    \item \label{claim_lagbi: ii} For fixed \(\mu\), the locus in the
\((t_1,t_2)\)-plane is described by parabolic arcs. More precisely, if
\(|p|>|q|\), then the relevant locus is the arc of a concave parabola lying
in the half-plane \(t_2\geq 0\). If \(|q|>|p|\), then the relevant locus is
the arc of a convex parabola lying in the half-plane \(t_2\leq 0\). When
\(|p|=|q|=0\), we obtain \(\mu=t_1=t_2=0\), corresponding to the single point
\((0,0)\).
    
    \item \label{claim_lagbi: iii} The set $f(P_L \cap \mu^{-1}(\lambda))$ is simply connected.
    
    \item \label{claim_lagbi: iv} The induced map $\widetilde{f}_\lambda$ is injective.
\end{enumerate}
\end{clm}

\begin{proof}[Proof of Claim \ref{clm: lagbi}]
\noindent\textbf{(i)} The relation \eqref{eq: lagbi} follows by adding the following equations, obtained from the definition of the system:
\begin{align}
    \mu^2 &= 4\pi^2(p_1^2 q_2^2 - 2p_1 p_2 q_1 q_2 + p_2^2 q_1^2), \label{mulagbi} \\
    4\pi^2 t_1^2 &= 4\pi^2(p_1^2 q_1^2 + 2p_1 p_2 q_1 q_2 + p_2^2 q_2^2). \label{ylagbi}
\end{align}

\vspace{1em}
\noindent\textbf{(ii)} Fix \(\mu\) and set
\[
m=\left|\frac{\mu}{2\pi}\right|.
\]
Assume first that \(|p|>0\) and \(|p|\geq |q|\). Then \(t_2\geq 0\).
Writing \(r=|p|^2\), we have \(|q|^2=r-t_2\). Substituting this into
\eqref{eq: lagbi}, we obtain
\begin{equation} \label{eq: par+}
    r t_2 = r^2 - m^2 - t_1^2.
\end{equation}
Thus, for fixed \(r\), the full equation defines a concave parabola. The part
relevant to the image of \(P_L\cap\mu^{-1}(\lambda)\) is the parabolic arc
\[
t_2=\frac{r^2-m^2-t_1^2}{r},
\qquad
-\sqrt{r^2-m^2}\leq t_1\leq \sqrt{r^2-m^2},
\]
which lies in the half-plane \(t_2\geq 0\).

Similarly, assume that \(|q|>0\) and \(|q|\geq |p|\). Then \(t_2\leq 0\).
Writing \(s=|q|^2\), we have \(|p|^2=s+t_2\). Substituting this into
\eqref{eq: lagbi}, we obtain
\begin{equation} \label{eq: par-}
    s t_2 = -s^2 + m^2 + t_1^2.
\end{equation}
Thus, for fixed \(s\), the full equation defines a convex parabola. The part
relevant to the image of \(P_L\cap\mu^{-1}(\lambda)\) is the parabolic arc
\[
t_2=\frac{-s^2+m^2+t_1^2}{s},
\qquad
-\sqrt{s^2-m^2}\leq t_1\leq \sqrt{s^2-m^2},
\]
which lies in the half-plane \(t_2\leq 0\).

In the case \(|p|=|q|=0\), equation \eqref{eq: lagbi} implies
\(\mu=t_1=0\), and \(t_2=0\) by definition, so the solution is the single
point \((0,0)\).

\begin{figure}[h!]
    \centering
    \includegraphics[width=0.65\textwidth]{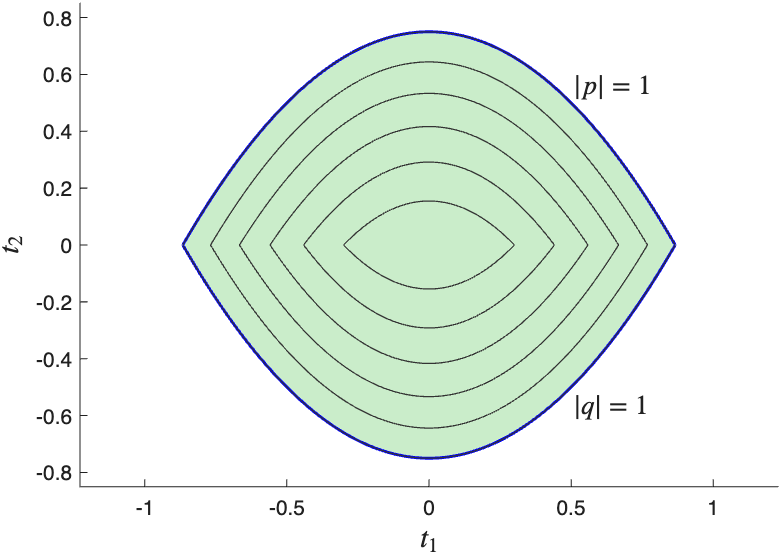}
    \caption{The image of \(P_L\cap\mu^{-1}(\lambda)\) under the map
    \(f=(t_1,t_2)\) for a fixed value of \(\mu\). The upper boundary is the
    parabolic arc corresponding to \(|p|=1\), while the lower boundary is the
    parabolic arc corresponding to \(|q|=1\). The shaded region is foliated by
    the parabolic arcs obtained by fixing \(|p|\) in the upper half and by
    fixing \(|q|\) in the lower half.}
    \label{fig:lagbid_parabolic_region}
\end{figure}

\vspace{1em}
\noindent\textbf{(iii)} The set
\(f(P_L\cap\mu^{-1}(\lambda))\cap\{t_2\geq 0\}\) is foliated by the
parabolic arcs described in \eqref{eq: par+}, with
\[
|p|^2\in\left[\left|\frac{\lambda}{2\pi}\right|,1\right].
\]
Similarly, the set
\(f(P_L\cap\mu^{-1}(\lambda))\cap\{t_2\leq 0\}\) is foliated by the
parabolic arcs described in \eqref{eq: par-}, with
\[
|q|^2\in\left[\left|\frac{\lambda}{2\pi}\right|,1\right].
\]
These two families of arcs meet along \(t_2=0\), and together they fill the
region between the two boundary arcs corresponding to \(|p|=1\) and
\(|q|=1\); see Figure~\ref{fig:lagbid_parabolic_region}. When
\[
|p|^2=|q|^2=\left|\frac{\lambda}{2\pi}\right|,
\]
it follows from \eqref{eq: lagbi} that \(t_1=0\), and clearly \(t_2=0\).
Therefore, the image reduces to the single point \((0,0)\) at the innermost
leaf. Hence the full image \(f(P_L\cap\mu^{-1}(\lambda))\) is simply
connected.

\vspace{1em}
\noindent\textbf{(iv)} To show that $\widetilde{f}_\lambda$ is injective, we first note that equations \eqref{eq: par+} and \eqref{eq: par-} determine $|p|$ and $|q|$ uniquely for each point $(t_1, t_2)$. We now show that each fiber of $\widetilde{f}_\lambda$ consists of exactly one $S^1$-orbit.

If $|p| = 0$, then the unique orbit is given by rotating the coordinates $(q_1, q_2)$ with fixed $|q|$. If $|p| \neq 0$, we may use the $S^1$-action to set $p_1 = 0$, and the values of $q_1$ and $q_2$ are then uniquely determined by:
\[
q_1 = -\frac{\mu}{2\pi p_2}, \qquad q_2 = \frac{t_1}{p_2}.
\]
Thus, the orbit is uniquely specified in all cases.

This concludes the proof.
\end{proof}

In this case, we do not immediately get a foliation by simple closed curves of $f(P_L \cap \mu^{-1}(\lambda))$, but since it is simply connected, we can obtain one. By the last part of the proof of Proposition \ref{generalmeth2}, we know that the singular fibers of our Lagrangian torus fibration are determined by the cases where $f^{-1}(p) \cap \mu^{-1}(\lambda) \ncong S^1$. The set of fixed points of the circle action related to $\mu$ only retracts when $p_1 = p_2 = q_1 = q_2 = 0$, and by the defining equations, this implies that $\mu = t_1 = t_2 = 0$. We can choose a foliation of $f(P_L \cap \mu^{-1}(0))$ given by simple closed curves that retract to a point different from $(0,0)$. 

Now we are going to obtain the shape of the base diagram as in
\cite[Section 2]{achigandrango2024singular}. For this, note that the boundary
of the base diagram does not change by changing the foliation; in fact, it is
given by the limit when \(\{|p|=1\}\cup\{|q|=1\}\). In the
\((t_1,t_2)\)-plane, these two boundary components are the two parabolic arcs
shown in Figure~\ref{fig:lagbid_parabolic_region}.

We now compute, for the boundary components \( \{|p| = 1\} \cup \{|q| = 1\} \), the symplectic area \( \mathcal{A}(\mu) \) of a disk inside \( P_L \) whose boundary lies in the limit Lagrangian projecting to the union of the curves defined by equations~\eqref{eq: par+} (for \( |p| = 1 \)) and~\eqref{eq: par-} (for \( |q| = 1 \)). These two curves join to form a closed, piecewise-smooth curve \( \gamma_\mu \).

Note that we may act with the \( S^1 \)-action on \( f^{-1}(p) \cap \mu^{-1}(\lambda) \) to fix \( q_1 = 0 \). There are exactly two loops in \( f^{-1}(\gamma_\mu) \cap \mu^{-1}(\lambda) \) satisfying \( q_1 = 0 \). We select one of them to compute the area of the corresponding disk.

For the part that projects to equation \eqref{eq: par+} with $|p|=1$, we obtain the parametrization given by $t$:
\begin{align*}
p_1&=\frac{\mu}{\sqrt{\mu^2+4\pi^2t^2}},\\
	p_2&=\frac{2\pi t}{\sqrt{\mu^2+4\pi^2t^2}},\\
	q_2&=\sqrt{\frac{\mu^2}{4\pi^2}+t^2},
\end{align*}
for $t\in \left[-\sqrt{1 - \left(\frac{\mu}{2\pi}\right)^2},\sqrt{1 - \left(\frac{\mu}{2\pi}\right)^2}\right]$.\\

For the part that projects to equation \eqref{eq: par-} with $|q|=1$, we obtain the parametrization given by $t$:
\begin{align*}
p_1&=\frac{\mu}{2\pi},\\
	p_2&=t,\\
	q_2&=1,
\end{align*}
	for $t\in \left[-\sqrt{1 - \left(\frac{\mu}{2\pi}\right)^2},\sqrt{1 - \left(\frac{\mu}{2\pi}\right)^2}\right]$.\\
	
	Now we calculate the symplectic area of the disc $D$ that has the cycle defined above as the boundary. The symplectic area is:
\begin{equation}\label{eq: bound_lagbi}
	\begin{aligned}
		\mathcal{A}(\mu)=\int_D \omega &=\int_{\partial D} (p_1dq_1+p_2dq_2)\\
		&=\int_{|p|=1} (p_1dq_1+p_2dq_2)+\int_{|q|=1} (p_1dq_1+p_2dq_2)\\
		&=\int_{|p|=1} p_2dq_2\\
		&=\int_{-\sqrt{1 - \left(\frac{\mu}{2\pi}\right)^2}}^{\sqrt{1 - \left(\frac{\mu}{2\pi}\right)^2}} \frac{4\pi^2t^2}{\mu^2+4\pi^2t^2}dt\\
		&=\left[t - \frac{\mu \cdot \arctan\left(\frac{2\pi \cdot t}{\mu}\right)}{2\pi}\right]_{-\sqrt{1 - \left(\frac{\mu}{2\pi}\right)^2}}^{\sqrt{1 - \left(\frac{\mu}{2\pi}\right)^2}}\\
		&=\frac{\sqrt{-\mu^2 + 4 \pi^2} - \mu \arctan \left(\frac{\sqrt{-\mu^2 + 4 \pi^2}}{\mu} \right)}{\pi}
	\end{aligned}
	\end{equation}
\end{proof}

\begin{rmk}
	To recover the result from \cite{ramos2017symplectic}, we perform the substitution $\mu = 2\pi \cos(\alpha/2)$ with $\alpha \in [0, 2\pi]$, and then apply the mutations as indicated in Figure~\ref{fig:convex_base_lagbi_2}. 
\end{rmk}

\section{Embedding results} \label{sec: per}

\subsection{Proof of Proposition \ref{prop: capacity_S^3}}

The idea to find embedded balls into $D^*S^3$ is by making use of Traynor's trick, introduced in \cite{Traynor1995SymplecticPC}, in the fibration obtained in Theorem \ref{thm: highellips}. The same approach could be used for the case of $D^*\mathcal{E}(1,1,c,c)$ as explained in Section \ref{sec: embres}. 

As remarked in \cite{Traynor1995SymplecticPC}, a weaker version of the Traynor trick holds in any dimension. In what follows we make a precise statement of such version and give a proof, for the sake of completeness.

\begin{definition}
Let 
\begin{align*}
D(r):&=\{(x,y)\in \mathbb{R}^2\mid x^2+y^2\leq r\},\\
SD(r):&=D(r)-\{x\geq 0,y=0\},\\
SD(r_1,\ldots,r_n):&=SD(r_1)\times\ldots\times SD(r_n),\\
\Delta^n(r):&=\{(x_1,\ldots,x_n)\in \mathbb{R}_{>0}^n\mid 0<x_1+\ldots+x_n<r\},\\
C^n(r):&=\{(y_1,\ldots,y_n)\in \mathbb{R}^n\mid 0<y_1,\ldots,y_n<r\}.
\end{align*}
\end{definition}

Notice that, the Lagrangian product $\Delta ^n(r)\times_L C^n(\pi)\subset \mathbb{C}^n$  has the same volume as the ball $B^{2n}(r)$ of radius $\sqrt{r}$.

The next lemma is due to McDuff, Polterovich and Traynor; and it is usually refered to as the Traynor trick.

\begin{lemma}\label{lem: traynor}
For every $\rho<r$, there exists a symplectic embedding $B^{2n}(\rho)\hookrightarrow \Delta^n(r)\times_L C^n(\pi)$.
\end{lemma}
\begin{proof}
Consider the map $\psi:\Delta^n(r)\times_L C^n(\pi)\rightarrow \textup{Int }B^{2n}(r)$ given by 
$$\psi(x_1,\ldots,x_n,y_1,\ldots,y_n)=(\sqrt{x_1}\cos(2y_1),-\sqrt{x_1}\sin(2y_1),\ldots,\sqrt{x_n}\cos(2y_n),-\sqrt{x_n}\sin(2y_n)).$$
It is not hard to see that this is a symplectic embedding and that 
$$\psi(\Delta^n(r)\times_L C^n(\pi))=\textup{Int }B^{2n}(r)\cap SD(r,\ldots,r).$$
Given $\rho<r$, choose an area-preserving diffeomorphism 
$$\sigma_{\rho}:D(\rho)\rightarrow SD(r),$$
such that, if $x^2+y^2\leq \alpha$ then $|\sigma_{\rho}(x,y)|^2\leq \alpha+(r-\rho)$. Hence, the map $\Psi_{\rho}:D(\rho)\times\cdots\times D(\rho)\rightarrow SD(r,\ldots,r)$ defined by
$$\Psi_{\rho}(x_1,y_1,\ldots,x_n,y_n)=(\sigma_{\rho}(x_1,y_1),\ldots,\sigma_{\rho}(x_n,y_n)),$$
is a symplectic embedding and $\Psi_{\rho}(B^{2n}(\rho))\subset \textup{Int }B^{2n}(r)\cap SD(r,\ldots,r)$. Therefore the map $\psi^{-1}\circ\Psi_{\rho}$ restricted to $B^{2n}(\rho)$ is a symplectic embedding into $\Delta^n(r)\times_L C^n(\pi)$. 
\end{proof}

\begin{lemma}\label{lem: embedding ball D*S^n}
For any $\varepsilon >0$ small enough, there exists a symplectic embedding $B^{6}(2\pi-\varepsilon)\hookrightarrow D^*S^3$.
\end{lemma}
\begin{proof}
Consider the pyramid $P_{\varepsilon} \subset \mathbb{R}^3$ defined as the convex hull of the following set of points:
\[
\left\{
\left(2\pi - \varepsilon, \frac{\varepsilon}{2}, \frac{\varepsilon}{4}\right),
\left(-2\pi + \varepsilon, \frac{\varepsilon}{2}, \frac{\varepsilon}{4}\right),
\left(0, 2\pi - \frac{\varepsilon}{2}, \frac{\varepsilon}{4}\right),
\left(0, \frac{\varepsilon}{2}, \pi - \frac{\varepsilon}{4}\right)
\right\}.
\]
Note that $P_{\varepsilon}$ is contained in the image of the moment map of the fibration described in Theorem~\ref{thm: highellips} for $c = 1$, since the moment image is the convex hull of 
\[
\left\{
\left(2\pi,0,0\right),
\left(-2\pi,0,0\right),
\left(0,2\pi,0\right),
\left(0,-2\pi,0\right),
\left(0,0, \pi\right)
\right\}.
\]

We now perform the transferring the cut operations as in \cite[Section 6.3]{achigandrango2024singular}. First, we apply the matrix
\[
\begin{pmatrix}
1 & 0 & 0 \\
0 & 1 & 0 \\
-1 & 0 & 1
\end{pmatrix}
\]
to the region defined by $\mu_1 < 0$. Then, we apply the matrix
\[
\begin{pmatrix}
1 & 0 & 0 \\
0 & 1 & 0 \\
0 & -1 & 1
\end{pmatrix}
\]
to the region defined by $\mu_2 < 0$. Finally, we apply the matrix
\[
\begin{pmatrix}
1 & 0 & 1 \\
0 & 1 & 0 \\
0 & 0 & 1
\end{pmatrix}
\]
to the entire image of the moment map.

As a result, the pyramid $P_{\varepsilon}$ is transformed into a new pyramid $\overline{P}_{\varepsilon}$, defined as the convex hull of the points:
\[
\left\{
\left(\frac{\varepsilon}{4}, \frac{\varepsilon}{2}, \frac{\varepsilon}{4}\right),
\left(2\pi - \frac{3\varepsilon}{4}, \frac{\varepsilon}{2}, \frac{\varepsilon}{4}\right),
\left(\frac{\varepsilon}{4}, 2\pi - \frac{\varepsilon}{2}, \frac{\varepsilon}{4}\right),
\left(\frac{\varepsilon}{4}, \frac{\varepsilon}{2}, 2\pi - \frac{3\varepsilon}{4}\right)
\right\}.
\]

By Lemma~\ref{lem: traynor}, for every $\delta > 0$ sufficiently small, there exists a symplectic embedding of the ball $B^6(2\pi - \varepsilon - \delta)$ into $\overline{P}_{\varepsilon} \times_L C^3(\pi)$. The claim then follows directly from the Symplectic Isotopy Extension Theorem due to Buhovsky and exposed in \cite{Schlenk2017SymplecticEP}, see their Lemma 8.1.
\end{proof}

We now prove the upper bound in Proposition \ref{prop: capacity_S^3}. For this, we define the Fermat quadric $Q^n$ in $\mathbb{C}P^{n+1}$ as 
$$Q^n=\{[z_0:\ldots:z_{n+1}]\in \mathbb{C}P^{n+1}\mid z_0^2+\ldots+z_{n+1}^2=0\}.$$
We also refer to $H_i\subset \mathbb{C}P^{n+1}$ as the hyperplane with vanishing $i$-th coordinate. It is not hard to see that the intersection $Q^n\cap H_{n+1}$ is naturally symplectomorphic to $Q^{n-1}$, and therefore, we decide to abuse the notation and denote this intersection as such. We next show how to identify $D^*S^{n}$ with $Q^n$, in a symplectic manner. The following proposition is folklore, we follow \cite{Oakley2013OnCL} and \cite{Adaloglou2022UniquenessOL}.

\begin{prop}\label{prop: embedding D*S^n into Q^n}
There exists a symplectomorphism $D^*S^n\hookrightarrow Q^n-Q^{n-1}$, where $Q^n-Q^{n-1}$ is considered with the restriction of the symplectic form $2\omega_{FS}$ induced by inclusion in $\mathbb{C}P^{n+1}$.
\end{prop}
\begin{proof}
We first claim there exists a symplectictomorphism from $B^{2n+2}(r)\subset (\mathbb{C}^{n+1},\omega_{std})$ to $(\mathbb{C}P^{n+1}-H_{n+1},r^2\omega_{FS})$. To see this, consider the map $\phi:B^{2n+2}(r)\rightarrow \mathbb{C}P^{n+1}-H_{n+1}$ given by 
$$\phi(z_0,\ldots,z_n)=\bigg[z_0:\ldots:z_n:i\sqrt{r^2-\sum|z_i|^2}\bigg].$$
The map $\phi$ factorizes trough maps $\widetilde{\phi}:B^{2n+2}(r)\hookrightarrow S^{2n+3}(r)\subset \mathbb{C}^{n+2}$ and $\pi_r:S^{2n+3}\rightarrow \mathbb{C}P^{n+1}$, where $\widetilde{\phi}$ is the map
$$\widetilde{\phi}(z_0,\ldots,z_n)=\bigg(z_0,\ldots,z_n,i\sqrt{r^2-\sum|z_i|^2}\bigg),$$
and $\pi_r$ the standard quotient by the scalar action. Notice that the natural inclusion $i_r:S^{2n+3}(r)\hookrightarrow \mathbb{C}^{n+2}$ satisfies
$$r^2\pi_r^*\omega_{FS}=i_r^*\omega_{std}.$$
From this we have that
$$\phi^*\omega_{FS}=\widetilde{\phi}^*\pi_r\omega_{FS}=\frac{1}{r^2}\widetilde{\pi}^*i_r^*\omega_{std}=\frac{1}{r^2}\omega_{std}.$$
The reason for the last equality to hold is given to the fact that the coordinate $z_{n+1}$ in the image of $\widetilde{\phi}$ is purely imaginary. The claim follows.

Note now that $D^*S^n$ is a subset of $B^{2n+2}(\sqrt{2})$. We show next that $\phi$ maps $D^*S^n$, seen as a subset of $B^{2n+2}(\sqrt{2})$, bijectively to $Q^n-Q^{n-1}$. To see this, let $[z_0:\ldots:z_n:z_{n+1}]\in \mathbb{C}P^{n+1}-H_{n+1}$, as $[z_0:\ldots:z_n:z_{n+1}]\in \phi(B^{2n+2}(\sqrt{2}))$, we have that
\begin{align*}
z_0^2+\ldots+z_n^2+z_{n+1}^2&=z_0^2+\ldots+z_n^2-2+\sum_{k=0}^n|z_k|^2,\\
&=\sum_{k=0}^n(x_k^2-y_k^2)+2i\sum_{k=0}^nx_ky_k-2+\sum_{k=0}^n(x_k^2+y_k^2),\\
&=2\bigg(\sum_{k=0}^nx_k^2-1\bigg)+2i\sum_{k=0}^nx_ky_k,
\end{align*}
and the result follows.
\end{proof}

\begin{proof}[Proof of Proposition \ref{prop: capacity_S^3}]
It follows from Lemma \ref{lem: embedding ball D*S^n} that $c_{Gr}(D^*S^3) \geq 2\pi$. On the other hand, by Proposition \ref{prop: embedding D*S^n into Q^n}, $D^*S^n$ embeds into $Q^n$ and given that $Q^n$ is a Hermitian symmetric space of compact type, we have by the main result in \cite{Loi2013SymplecticCO} that $c_{Gr}(Q^n) = 2\pi$. Then by monotonicity of symplectic capacities we have that $c_{Gr}(D^*S^n) \leq 2\pi$, and the result follows. 
\end{proof}

\begin{rmk}
It follows trivially from Proposition \ref{prop: ellipsoids}, and this was already observed in \cite{Ferreira2021SymplecticEI}, that there is a full packing of $D^*S^2$ by two symplectic balls of equal size. This, together with results in \cite{Gromov1985PseudoHC}, \cite{Mcduff1994SymplecticPA} and \cite{Biran1996SymplecticPI}, imply there is also a full packing of $D^*S^2$ by 8 or $2n$ balls of equal size, for $n\geq 9$. We can also see there are very full packings of $D^*S^2$ by 2, 5 and $n$ balls for $n\geq 7$.
\end{rmk}

\bibliographystyle{amsalpha}
\bibliography{Lagfib}

\providecommand{\bysame}{\leavevmode\hbox to3em{\hrulefill}\thinspace}
\providecommand{\MR}{\relax\ifhmode\unskip\space\fi MR }
\providecommand{\MRhref}[2]{%
  \href{http://www.ams.org/mathscinet-getitem?mr=#1}{#2}
}
\providecommand{\href}[2]{#2}
\begin{thebibliography}{LMZ15}

\bibitem[AA24]{achigandrango2024singular}
Santiago Achig-Andrango, \emph{Singular {L}agrangian torus fibrations on the smoothing of algebraic cones}, Journal of Symplectic Geometry \textbf{22} (2024), no.~4, 847--913.

\bibitem[Ada22]{Adaloglou2022UniquenessOL}
Nikolaos Adaloglou, \emph{Uniqueness of {L}agrangians in {$T^*\mathbb{R}P^2$}}, Annales {M}ath{\'e}matiques du Qu{\'e}bec \textbf{49} (2022), no.~1, 215--222.

\bibitem[Aur07]{auroux2007mirror}
Denis Auroux, \emph{Mirror symmetry and {$T$}-duality in the complement of an anticanonical divisor}, Journal of G{\"o}kova Geometry Topology \textbf{1} (2007), 51--91.

\bibitem[Bir96]{Biran1996SymplecticPI}
Paul Biran, \emph{Symplectic packing in dimension 4}, Geometric \& Functional Analysis GAFA \textbf{7} (1996), 420--437.

\bibitem[CNS02]{cushman2002sign}
Richard Cushman and Vu~Ngoc~San, \emph{Sign of the monodromy for {L}iouville integrable systems}, Annales Henri Poincar{\'e}, vol.~3, Springer, 2002, pp.~883--894.

\bibitem[Efs05]{efstathiou2005metamorphoses}
Konstantinos Efstathiou, \emph{Metamorphoses of {H}amiltonian systems with symmetries}, Springer, 2005.

\bibitem[FR22]{Ferreira2021SymplecticEI}
Brayan Ferreira and Vinicius Gripp~Barros Ramos, \emph{Symplectic embeddings into disk cotangent bundles}, Journal of Fixed Point Theory and Applications \textbf{24} (2022), 62.

\bibitem[FRV26]{Ferreira2023GromovWO}
Brayan Ferreira, Vinicius Gripp~Barros Ramos, and Alejandro Vicente, \emph{Gromov width of disk cotangent bundles of spheres of revolution}, Advances in Mathematics \textbf{487} (2026), 110761.

\bibitem[Gro85]{Gromov1985PseudoHC}
Mikhael Gromov, \emph{Pseudo holomorphic curves in symplectic manifolds}, Inventiones {M}athematicae \textbf{82} (1985), 307--347.

\bibitem[Gro01]{gross2001examples}
Mark Gross, \emph{Examples of special {L}agrangian fibrations}, Symplectic geometry and mirror symmetry, World Scientific, 2001, pp.~81--109.

\bibitem[HK21]{HohlochKarshon2021}
Sebastian Hohloch and Yael Karshon, \emph{Extending compact hamiltonian {$S^1$}-spaces to integrable systems with mild degeneracies in dimension four}, arXiv:2105.00523 (2021).

\bibitem[KS18]{Kislev2018BoundsOS}
Asaf Kislev and Egor Shelukhin, \emph{Bounds on spectral norms and barcodes}, Geometry \& Topology \textbf{22} (2018), no.~6, 3255--3281.

\bibitem[KT01]{KarshonTolman2001}
Yael Karshon and Susan Tolman, \emph{Centered complexity one hamiltonian torus actions}, Transactions of the American Mathematical Society \textbf{353} (2001), no.~12, 4831--4861.

\bibitem[KT14]{KarshonTolman2014}
\bysame, \emph{Classification of hamiltonian torus actions with two-dimensional quotients}, Geometry \& Topology \textbf{18} (2014), no.~2, 669--716.

\bibitem[LFP24]{LeFlochPalmer2024}
Yohann Le~Floch and Joseph Palmer, \emph{Semitoric families}, vol. 302, Memoirs of the American Mathematical Society, no. 1514, American Mathematical Society, 2024.

\bibitem[Lio55]{Liouville1855}
Joseph Liouville, \emph{Note sur l'int{\'e}gration des {\'e}quations diff{\'e}rentielles de la dynamique}, Journal de Math{\'e}matiques Pures et Appliqu{\'e}es \textbf{20} (1855), 137--138.

\bibitem[LMS11]{Latschev2011TheGW}
Janko Latschev, Dusa Mcduff, and Felix Schlenk, \emph{The {G}romov width of 4-dimensional tori}, Geometry \& Topology \textbf{17} (2011), 2813--2853.

\bibitem[LMZ15]{Loi2013SymplecticCO}
Andrea Loi, Roberto Mossa, and Fabio Zuddas, \emph{Symplectic capacities of {H}ermitian symmetric spaces of compact and noncompact type}, Journal of Symplectic Geometry \textbf{13} (2015), no.~4, 1049--1073.

\bibitem[MP94]{Mcduff1994SymplecticPA}
Dusa Mcduff and Leonid Polterovich, \emph{Symplectic packings and algebraic geometry}, Inventiones {M}athematicae \textbf{115} (1994), 405--429.

\bibitem[ORS25]{Ostrover2023FromLP}
Yaron Ostrover, Vinicius G.~B. Ramos, and Daniele Sepe, \emph{From {L}agrangian products to toric domains via the {T}oda lattice}, Compositio Mathematica \textbf{161} (2025), no.~2, 365--384.

\bibitem[OU13]{Oakley2013OnCL}
Joel Oakley and Michael Usher, \emph{On certain {L}agrangian submanifolds of {$S^2 \times S^2$} and {$\mathbb{C}P^n$}}, Algebraic \& Geometric Topology \textbf{16} (2013), 149--209.

\bibitem[Ram17]{ramos2017symplectic}
Vinicius Gripp~Barros Ramos, \emph{Symplectic embeddings and the {L}agrangian bidisk}, Duke Mathematical Journal \textbf{166} (2017), no.~9, 1703--1738.

\bibitem[RS19]{Ramos2017OnTR}
Vinicius G.~B. Ramos and Daniele Sepe, \emph{On the rigidity of {L}agrangian products}, Journal of Symplectic Geometry \textbf{17} (2019), no.~5, 1447--1478.

\bibitem[Sch17]{Schlenk2017SymplecticEP}
Felix Schlenk, \emph{Symplectic embedding problems, old and new}, Bulletin of the American Mathematical Society \textbf{55} (2017), 139--182.

\bibitem[Sym]{symington71four}
Margaret Symington, \emph{Four dimensions from two in symplectic topology. {T}opology and geometry of manifolds,(athens, ga, 2001), 153--208}, Proc. Sympos. Pure Math, vol.~71.

\bibitem[Tra95]{Traynor1995SymplecticPC}
Lisa Traynor, \emph{Symplectic packing constructions}, Journal of Differential Geometry \textbf{42} (1995), 411--429.

\end{thebibliography}

\end{document}